\documentclass{scrartcl}

\pdfoutput=1

\usepackage[utf8]{inputenc}
\usepackage[UKenglish]{babel}
\usepackage{amsmath}
\usepackage{amsthm}
\usepackage{amssymb}
\usepackage{mathtools}
\usepackage[mathscr]{eucal}
\usepackage{color}
\usepackage{bussproofs}
\usepackage{verbatim}
\usepackage{graphics}
\usepackage{xparse}

\definecolor{darkred}{rgb}{0.5,0,0}

\usepackage[linkcolor=darkred,citecolor=blue,colorlinks,pdfpagelabels,plainpages=false]{hyperref}

\theoremstyle{plain}
\newtheorem{Lma}{Lemma}[section]
\newtheorem{Prp}[Lma]{Proposition}
\newtheorem{Thm}[Lma]{Theorem}
\newtheorem{Cor}[Lma]{Corollary}

\theoremstyle{definition}
\newtheorem{Def}[Lma]{Definition}%[section]

\theoremstyle{remark}
\newcommand{\N}{\ensuremath{\mathbb{N}}}

\newcommand{\falsum}{\bot}

\newcommand{\limp}{\rightarrow}
\newcommand{\liff}{\leftrightarrow}

\newcommand{\bigland}{\bigwedge}

\renewcommand{\=}{\ensuremath{\mathrel{=\!\!\!\!=}}}
\newcommand{\0}{\mathsf{0}}
\newcommand{\nonequality}{\mathrel{=\!\!\!\!\not=}}

\newcommand{\num}[1]{\ensuremath{\overline{#1}}}

\newcommand{\GN}[1]{\ensuremath{\ulcorner #1 \urcorner}}
\newcommand{\sep}{\; | \;}
\newcommand{\PA}{\ensuremath{\mathrm{PA}}}

\NewDocumentCommand{\prf}{ O{} }{\mathrel{\hbox{%
    \oalign{$\vdash$\cr%
            \noalign{\kern-2.5ex}%
            \hfil$\scriptstyle\thinspace #1$\hfil\cr}%
    }}}

\renewcommand{\L}{\ensuremath{\mathcal{L}}}

\DeclareMathOperator{\FV}{FV}

\DeclareMathOperator{\val}{val}

\DeclareMathOperator{\I}{I}
\DeclareMathOperator{\POS}{POS}

\newcommand{\TND}{\ensuremath{\mathrm{LEM}}}
\newcommand{\LPRA}{\ensuremath{\L_{\PRA}}}
\newcommand{\LHA}{\ensuremath{\L_{\HA}}}
\newcommand{\LHAP}{\ensuremath{\L_{\HAP}}}
\newcommand{\LID}{\ensuremath{\L(\mathrm{ID})}}
\newcommand{\LP}{\ensuremath{\L(P)}}
\newcommand{\LIDHA}{\ensuremath{\LHA(\mathrm{ID})}}
\newcommand{\LIDHAP}{\ensuremath{\LHAP(\mathrm{ID})}}
\newcommand{\LPHA}{\ensuremath{\LHA(P)}}
\newcommand{\LPHAP}{\ensuremath{\LHAP(P)}}

\newcommand{\ID}[1]{\ensuremath{\widehat{\mathrm{ID}}_{#1}}{}}
\newcommand{\IID}[1]{\ensuremath{\widehat{\mathrm{ID}}{}_{#1}^{\mathrm{i}}{}}}
\newcommand{\IIDP}[1]{\ensuremath{\widehat{\mathrm{ID}}{}_{#1}^{\mathrm{i}}\mathrm{P}{}}}
\newcommand{\HA}{\ensuremath{\mathrm{HA}}}
\newcommand{\PRA}{\ensuremath{\mathrm{PRA}}}
\newcommand{\PRHA}{\ensuremath{\HA_{\LPRA}}}
\newcommand{\PRPA}{\ensuremath{\PA_{\LPRA}}}

\newcommand{\HAP}{\ensuremath{\mathrm{HAP}}}
\newcommand{\PAP}{\ensuremath{\mathrm{PAP}}}

\newcommand{\LPT}{\ensuremath{\mathrm{LPT}}}

\newcommand{\KF}{\ensuremath{\mathrm{KF}}}

\newcommand{\bIID}[1]{\ensuremath{\mathbf{\widehat{ID}{}_{#1}^{i}}{}}}
\newcommand{\bIIDP}[1]{\ensuremath{\mathbf{\widehat{ID}{}_{#1}^{i}P}{}}}

\newcommand{\bHAP}{\ensuremath{\mathbf{HAP}}}
\newcommand{\bLPT}{\ensuremath{\mathbf{LPT}}}

\newcommand{\AnHier}{\ensuremath{\Lambda}}

\newcommand{\AnHierN}{\ensuremath{\AnHier^{\mathrm{nf}}}}

\newcommand{\AnHierC}{\AnHier}
\newcommand{\bAnHierC}{\ensuremath{\mathbf{\bAnHierC}}}
\newcommand{\Ax}{\ensuremath{\mathrm{Ax}}}
\newcommand{\rea}{\mathbin{\underline{\mathrm{r}}}}
\newcommand{\reha}{\mathbin{\underline{\hat{\mathrm{r}}}}}

\DeclareMathOperator{\conjs}{conjs}

\newcommand{\Sigge}{\ensuremath{\Sigma\mathrm{e}}}

\newcommand{\den}{\ensuremath{\mathord\downarrow}}
\newcommand{\app}{\ensuremath{\cdot}}
\newcommand{\cok}{\ensuremath{\mathbf{k}}}
\renewcommand{\cos}{\ensuremath{\mathbf{s}}}
\newcommand{\cop}{\ensuremath{\mathbf{p}}}
\newcommand{\copl}{\ensuremath{\mathbf{p}_{\mathrm{l}}}}
\newcommand{\copr}{\ensuremath{\mathbf{p}_{\mathrm{r}}}}
\newcommand{\cosucc}{\ensuremath{\mathbf{succ}}}
\newcommand{\cor}{\ensuremath{\mathbf{r}}}
\newcommand{\cod}{\ensuremath{\mathbf{d}}}
\newcommand{\coid}{\ensuremath{\mathbf{id}}}

\newcommand{\cofix}{\ensuremath{\mathbf{fix}}}
\newcommand{\comin}{\ensuremath{\mathbf{min}}}

\DeclareMathOperator{\rec}{rec}

\newcommand{\fv}{\ensuremath{\mathsf{v}}}

\newcommand{\fdsub}{\ensuremath{\mathtt{sub}}}

\newcommand{\fdval}{\ensuremath{\mathtt{val}}}

\newcommand{\fdSat}{\ensuremath{\mathtt{Sat}}}

\newcommand{\fdAnHierN}[1][]{\ensuremath{\mathtt{\AnHier}_{#1}^{\mathrm{nf}}}}

\newcommand{\fdlambda}{\ensuremath{\mathtt{\lambda}}}

\newcommand{\f}[1]{\ensuremath{\mathsf{#1}}}

\newcommand{\fZ}{\f{Z}}
\newcommand{\fS}{\f{S}}

\newcommand{\fdVar}{\ensuremath{\mathtt{Var}}}

\newcommand{\fdTerm}{\ensuremath{\mathtt{Term}}}

\newcommand{\fdFV}{\ensuremath{\mathtt{FV}}}

\newcommand{\fdBV}{\ensuremath{\mathtt{BV}}}

\newcommand{\fdconjs}{\ensuremath{\mathtt{conjs}}}

\newcommand{\fdSeq}{\ensuremath{\mathtt{Seq}}}

\newcommand{\fdlh}{\ensuremath{\mathtt{lh}}}

\newcounter{memory}
\allowdisplaybreaks

\title{Revisiting the conservativity of fixpoints over intuitionistic arithmetic}
\author{Mattias Granberg Olsson\\Graham E.\ Leigh}
\date{University of Gothenburg}

\begin{document}

\maketitle

\vspace{-3 ex}

\begin{abstract}
  This paper presents a novel proof of the conservativity of the intuitionistic theory of strictly positive fixpoints, $\IID1$, over Heyting arithmetic ($\HA$), originally proved in full generality by Arai (2011). The proof embeds $\IID1$ into the corresponding theory over Beeson's logic of partial terms and then uses two consecutive interpretations, a realizability interpretation of this theory into the subtheory generated by almost negative fixpoints, and a direct interpretation into Heyting arithmetic with partial terms using a hierarchy of satisfaction predicates for almost negative formulae. It concludes by applying van den Berg and van Slooten's result (2018) that Heyting arithmetic with partial terms plus the schema of self realizability for arithmetic formulae is conservative over $\HA$.

\end{abstract}

\pagenumbering{arabic}

% -------------------------------
% -------------------------------
\section{Introduction}
\label{sec:Intro}
% -------------------------------
% -------------------------------

An occurrence of a symbol in a formula $\varphi$ is \emph{strictly positive} if the occurrence is not inside an antecedent of an implication in $\varphi$. Given a parameter predicate $P$ we can thus form all formulae which have only strictly positive occurrences of $P$, referred to as the \emph{strictly positive operator forms}. Given such an operator form \( \Phi(P;x) \) with free variable \( x \) and writing \( \Phi(\varphi;x) \) for the result of substituting \( \varphi \) for \( P \) in \( \Phi \), a predicate $\I_{\Phi}$ satisfying
\begin{align}\label{eq:IDAx}
  \forall x ( \I_{\Phi}(x) \liff \Phi(\I_{\Phi};x) )
\end{align}
is then a \emph{fixpoint} of the operator form $\Phi$. The theory $\ID1$ is Peano arithmetic ($\PA$) extended with a new symbol $\I_{\Phi}$ and axiom \eqref{eq:IDAx} for all strictly positive operator forms $\Phi$, and $\IID1$ is the corresponding extension of Heyting arithmetic ($\HA$). Inductively, $\ID{n + 1}$ and $\IID{n + 1}$ are the corresponding extensions of $\ID{n}$ and $\IID{n}$ respectively. $\ID1$ is not conservative over $\PA$, since already $\ID1(\Pi_2)$, in which \eqref{eq:IDAx} holds only for strictly positive $\Pi_2$ operator forms $\Phi$, proves the consistency of $\PA$. On the contrary, $\IID1$ is known to be conservative over $\HA$.

In the 1997 paper~\cite{Buchholz:1997}, Wilfried Buchholz proved that the theory of fixpoints for \emph{strongly} positive operator forms, that is where \eqref{eq:IDAx} holds for operator forms $\Phi$ which contain no implications whatsoever, is conservative over $\HA$ for almost negative sentences. The result was soon improved to full conservativity by Toshiyasu Arai in~\cite{Arai:1998}. A few years later (2002), Christian R{\"u}ede and Thomas Strahm in~\cite{Ruede_Strahm:2002} made an improvement in another direction, by showing that the theory $\IID{\alpha}$, where $\alpha$ ($< \Gamma_0$) refers to ordinal iterations of the fixpoint construction, is conservative over the theory $\mathrm{ACA}^{-i}_{\alpha}$ (intuitionistic theory of iterated arithmetical comprehension without set parameters) for negative and $\Pi^0_2$ sentences. Their argument makes a realizability interpretation of $\IID{\alpha}$ in an appropriate fragment, which is subsequently interpreted in the classical theory $\mathrm{ACA}^{-}_{\alpha}$ via partial truth predicates; finally $\mathrm{ACA}^{-}_{\alpha}$ is conservative over $\mathrm{ACA}^{-i}_{\alpha}$ for negative and $\Pi^0_2$ sentences. As a corollary, $\IID{n}$ is conservative over $\HA$ for such sentences.

A few years after this series of partial results (2011), Arai in \cite{Arai:2011} finally showed, using cut-elimination of an infinitary derivation system formalised in $\HA$, that the theory $\IID1$ is fully conservative over $\HA$ (the paper actually shows that all $\IID{n}$ are conservative over $\HA$).

The aim of this paper is to reprove this important result by a new method. For clarity we first outline our argument for conservation of almost negative sentences only. This closely resembles that of R{\"u}ede and Strahm outlined above in that it combines two natural translations: a realizability interpretation of \( \IID{1} \) in the subtheory with the fixpoint axiom \eqref{eq:IDAx} restricted to almost negative operator forms, and a direct interpretation of this subtheory in Heyting arithmetic. As in \cite{Ruede_Strahm:2002}, the first step uses standard Kleene-style realizability in the manner of Buchholz' \cite[§6]{Buchholz:1981}. The main difference lies in the second reduction, which interprets fixpoint predicates for almost negative operator forms by partial satisfaction predicates definable in Heyting arithmetic and generalises the conservativity of fixpoints for strictly positive \( \Pi_1 \) operator forms over Peano arithmetic. In this way, the detour through classical logic used in \cite{Ruede_Strahm:2002} can be avoided and conservativity of $\IID1$ over $\HA$ obtained also for almost negative sentences outside $\Pi^0_2$, since, combining the two interpretations, if \( \varphi \) is an arithmetic theorem of \( \IID{1} \) we obtain that \( \varphi \) is realizable in \( \HA \) and thus \( \HA \vdash \varphi \) if \( \varphi \) is almost negative.

Full conservation can be obtained by the same argument if the theories in question are formulated over Beeson’s logic of partial terms (see~\cite{Beeson:1985}) rather than intuitionistic predicate logic. 
Let \( \HAP \) and \( \IIDP1 \) be Heyting arithmetic with partial terms and its extension to fixpoints for strictly positive operator forms respectively.
We show that every arithmetic theorem of \( \IID1 \) is realized in \( \HAP \) via an analogous combination of interpretations.
This involves a realizability interpretation of \( \IIDP1 \) in the subtheory of almost negative fixpoints, which we then show is a conservative extension of \( \HAP \). This conservativity is likewise witnessed by the definability of partial satisfaction predicates for almost negative formulae involving partial terms. 
The final step of concluding \( \HA \vdash \varphi \) from the realizability of \( \varphi \) in \( \HAP \) is a consequence of a result due to Benno van den Berg and Lotte van Slooten in \cite{vandenBerg_vanSlooten:2018} that for this realizability interpretation, Heyting arithmetic is conservatively extended by \( \HAP + \varphi \leftrightarrow \exists x (x \rea \varphi) \) for arithmetic $\varphi$, where \( x \rea \varphi \) expresses `\( x \) realizes \( \varphi \).'

This paper constitutes part of the first authors thesis, and the details of the omitted proofs will be presented in the dissertation \cite{MGO:202?}.

\subsection{Outline}
\label{sec:disp}

We will start by fixing basic notation and terminology in section \ref{sec:prel}. 
Heyting arithmetic with partial terms and the basic realizability interpretation we utilise is introduced in section \ref{sec:HAP}. The main results from \cite{vandenBerg_vanSlooten:2018} which we require are also rehearsed in that section. 
Section~\ref{sec:Lambda} concerns properties of the almost negative formulae. We present a hierarchy of formulae \( ( \AnHierC_n )_n \) based on quantifier complexity which exhausts the almost negative formulae and present for each \( n \) a \( \AnHierC_n \)-formula that provably in \( \HAP \) is a satisfaction predicate for \( \AnHierC_n \)-formulae.
Section~\ref{sec:Fix} overviews three intuitionistic fixpoint theories of import: the theory \( \IID{1} \) of strictly positive fixpoints over Heyting arithmetic, the counterpart theory over the logic partial terms, \( \IIDP{1} \), and its subtheory of fixpoints for almost negative operator forms, \( \IIDP{1}(\AnHierC) \), as well as interpretability results among them. The article concludes with proving the main result and a discussion of the methods and potential extensions, in section \ref{sec:Conc}.

% -------------------------------
% -------------------------------
\section{Preliminaries}
\label{sec:prel}
% -------------------------------
% -------------------------------

Our base languages will be $\LHA$, $\LPRA$ and $\LHAP$, all of which will contain the quantifiers $\forall$ and $\exists$, the connectives $\land$, $\lor$ and $\limp$, the propositional constant $\falsum$ and the equality relation $\=$ as logical symbols. $\LHA$, the language of Heyting arithmetic, in addition contains the constant symbol $\0$ and the function symbols $\fS$, $+$ and $\times$. Numerals in $\LHA$ will be constructed from $\0$ and $\fS$ in the usual way ($\num{0} = \0$, $\num{n + 1} = \fS(\num{n})$); the same holds for all languages we will consider. $\LPRA$, the language of primitive recursive arithmetic, contains the symbols of $\LHA$ as well as a function symbol for each primitive recursive function (including fresh symbols for addition and multiplication). The final language, $\LHAP$, of Heyting arithmetic with partial terms extends $\LHA$ by constant symbols $\cok$, $\cos$, $\copl$, $\copr$, $\cop$, $\cosucc$, $\cor$ and the binary function symbol $\app$. Since \( \app \) will occur more frequently than $\times$ we will often abbreviate it by juxtaposition and let $\times$ be written out. $\HA$ ($\PA$) is the intuitionistic (resp.\ classical) theory in $\LHA$ axiomatised by basic axioms for the successor $\fS$, defining equations for $+$ and $\times$, and induction for $\LHA$. Heyting arithmetic with partial terms, $\HAP$, is the theory in $\LHAP$ presented in \cite{vandenBerg_vanSlooten:2018} which we will return to in section \ref{sec:HAP}. $\HA(\L)$ and $\HAP(\L)$ will be $\HA$ and $\HAP$ with induction extended (or restricted) to $\L$, but with no additional axioms concerning the non-arithmetical symbols. Finally, $\PRHA$ is Heyting arithmetic axiomatised in $\LPRA$ with defining equations for all function symbols as axioms, essentially as in, for example, Troelstra and van~Dalen's \cite{Troelstra_vanDalen:1988a}. While there is no real difference between $\HA$ and $\PRHA$ (they are mutually interpretable), the small number of symbols and the fact that $\LHA \subseteq \LHAP$ will be advantageous from our perspective; at the same time, several formula-classes are less expressive in $\HA$ than in $\PRHA$, which makes \( \PRHA \) more desirable to work with in some situations.

We use an infinite sequence of variables, $\fv_0, \fv_1, \dotsc$; metavariables for these will be denoted by $x,y,z,u,v,w,\dotsc$ etc.. Terms are constructed from variables, constants and function symbols in the usual way; for disambiguation we will consider a variable and the term made up of only that variable as distinct syntactic objects, in particular when it comes to coding. We use $\tau,\sigma,\rho,\dotsc$ as metavariables for terms. We denote tuples of terms (or variables) with $\vec{\tau}$, or $\vec{\tau}^n$ if we want to emphasise the tuple has length $n$. Symbols $\Phi, \Psi, \phi,\varphi,\psi,\theta,\vartheta,\dotsc$ denotes formulae, where uppercase is used to emphasise that the formula is an operator form, in that it contains  distinguished predicate symbol (denoted by $P$, $Q$ or $R$). Formulae are identified up to $\alpha$-equivalence, and we define the Gödel code of a formula to be the (numerically) least code of any of the $\alpha$-equivalent formulae with no \emph{nested bindings}, i.e.\ two nested quantifiers binding the same variable. Regarding substitution: $\sigma[\vec{u}/\vec{\tau}]$, where $\sigma$ is a term, $\vec{u}$ is a finite sequence of variables and $\vec{\tau}$ is a sequence of terms of the same length, denotes simultaneous substitution of $\tau_i$ for $u_i$ in $\sigma$ for all indices $i$ of the sequences. Similarly $\varphi(\vec{u}/\vec{\tau})$, where $\varphi$ is a formula and $\vec{u}$ and $\vec{\tau}$ are as above, denotes simultaneous substitution of $\tau_i$ for the \emph{free} occurrences of $u_i$ in $\varphi$ for all indices $i$ of the sequences \emph{after renaming bound variables} in $\varphi$ to avoid conflicts. Since we identify $\alpha$-equivalent formulae this is permissible. The simpler expression $\varphi(\vec{\tau})$ means $\varphi(\vec{\fv}_{<n}/\vec{\tau})$, where $\vec{\fv}_{<n}$ is $\langle \fv_0, \dotsc, \fv_{n - 1}\rangle$, with $n$ the length of $\vec{\tau}$. When introducing a formula with a formulation of the kind `Let $\varphi(\vec{x})$ be a formula,' it is tacitly understood that $\vec{x}^n$  is without repetitions and no variables other than $\fv_0,\dotsc,\fv_{n - 1}$ occurs free in $\varphi$, that is all free occurrences of $\vec{x}$ in $\varphi(\vec{x})$ are obtained by the substitution; $\varphi(\vec{\tau})$ is subsequently the same as $\varphi(\vec{x})(\vec{x}/\vec{\tau})$. We use the corresponding notation for substitution of formulae for relations: $\Phi(R/\lambda \vec{x}. \vartheta)$, where the length of $\vec{x}$ is the arity of $R$, denotes simultaneous substitution in the formula $\Phi$ of the formula $\vartheta$ for the relation symbol $R$, where the $i$th argument of $R$ is substituted for the free occurrences of $x_i$ in $\vartheta$, after renaming bound variables in $\Phi$ and $\vartheta$.\footnote{It is of course important that no variable occurring as an argument to $R$ is bound in $\vartheta$.} Note, we do not require that $x_i$ is a free variable in $\vartheta$ or that all free variables of $\vartheta$ are among $\vec{x}$. If we omit the `$\lambda \vec{x}$' we mean that $\vec{x}$ is $\vec{\fv}_{<n}$ and if $\vartheta$ consists of a single $n$-ary relation symbol $Q$ we write $\Phi(R/Q)$ for $\Phi(R/\lambda \vec{v}_{< n}. Q(\vec{v}_{< n}))$, where we might also omit $R$ if it is clear from the context (as in section \ref{sec:Fix}). Finally, when expressing both relation and term substitution, denoted \( \Phi(\lambda \vec{x}. \vartheta;\vec{\tau}) \), term substitution takes precedence, namely  $\Phi(\lambda \vec{x}. \vartheta;\vec{\tau}) = (\Phi(\vec{\tau}))(\lambda \vec{x}. \vartheta)$.\footnote{Note that all choices involved can be made in an unambiguous (primitive recursive) way, which we will assume without specifying one.} We will abuse notation and write $\chi \in \L$ for any symbol, term or formula $\chi$ of the language $\L$. The Gödel code of an expression $\eta$ (in the language under consideration) will be denoted by $\GN{\eta}$. When using numerals of Gödel codes we suppress the bar, so that in e.g.\ $\HA \prf \GN{\eta} \nonequality \0$, `$\GN{\eta}$' refers to the numeral of the number $\GN{\eta}$.

We finally turn to notations for syntactic formula-classes. $\AnHierC$ will denote the class of almost negative formulae; a formula is \emph{almost negative} if it contains no disjunctions and existential quantifiers only occur immediately in front of term-equations. Note the unfortunate clash of terminology that some strictly positive operator forms (meaning the distinguished predicate \( P \) does not occur in the antecedent of an implication) will at the same time be almost negative formulae. Both of these concepts are standard in their fields, so there is little point in trying to avoid or change them here. \emph{Negative} formulae are the almost negative formulae which contains no existential quantifiers whatsoever. When concerned with a syntactic formula-class \( \Gamma \) such as $\AnHierC$ or $\Sigma_n$ (or $\AnHier_n$ introduced in section~\ref{sec:Lambda}) we will use the convention that \( \Gamma \) refers to the set of formulae belonging to this class in the language under consideration, while $\Gamma(\HA)$ refers to the set of formulae which are $\HA$-equivalent to formulae from $\Gamma$ in this language; the same principle applies to other theories like $\HAP$ and $\PA$. For example $\Sigma_1(\PA)$ is the set of formulae in $\LHA$ which are $\PA$-equivalent to $\Sigma_1$-formulae. Some caution must be used here since $\LHA$ does not contain the symbol $<$; $\Delta_0$-formulae (and subsequently the rest of the arithmetical hierarchy) are thus defined via the defined inequality: $x < y$ is $\exists z (x + \fS(z) \= y)$. If $\HA^{\prime}$ is Heyting arithmetic in $\LHA + \mathord <$, this yields $\HA^{\prime}$ as a definitional extension of $\HA$, with the arithmetical hierarchy in $\HA$ the image of the one in $\HA^{\prime}$.

% -------------------------------
% -------------------------------
\section[$\LPT$ and $\HAP$]{$\bLPT$ and $\bHAP$}
\label{sec:HAP}
% -------------------------------
% -------------------------------

For completeness and convenience we here briefly describe the systems $\LPT$ and $\HAP$, essentially as given in \cite{vandenBerg_vanSlooten:2018}. Most proofs will be omitted; we refer the reader to the references for details.

\begin{Def}[LPT]\label{def:LPT}
  Logic of Partial Terms, $\LPT$, is a Hilbert-style deductive system in a language containing at least equality with the following logical rules and axiom schemata, where $\tau\den$ abbreviates $\tau \= \tau$ and expresses that \emph{$\tau$ denotes}:\\
  \begin{minipage}[t]{0.52\linewidth}\renewcommand{\=}{\ensuremath{{\,=\!\!\!\!=\,}}}
    \vspace{0.5\lineskip}
    \begin{enumerate}
    \item[\refstepcounter{enumi}\LPT\arabic{enumi}.\label{ax:Id}] $\phi \limp \phi$
    \item[\stepcounter{enumi}\refstepcounter{enumi}\LPT\arabic{enumi}.\label{rule:Trans}] $\phi \limp \psi, \psi \limp \chi \Rightarrow \phi \limp \chi$
    \item[\stepcounter{enumi}\refstepcounter{enumi}\LPT\arabic{enumi}.\label{rule:Con}] $\phi \limp \psi, \phi \limp \chi \Rightarrow \phi \limp \psi \land \chi$
    \item[\stepcounter{enumi}\refstepcounter{enumi}\LPT\arabic{enumi}.\label{rule:Dis}] $\phi \limp \chi, \psi \limp \chi \Rightarrow \phi \lor \psi \limp \chi$
    \item[\stepcounter{enumi}\refstepcounter{enumi}\LPT\arabic{enumi}.\label{rule:Uncurry}] $\phi \limp (\psi \limp \chi) \Rightarrow \phi \land \psi \limp \chi$
    \item[\stepcounter{enumi}\refstepcounter{enumi}\LPT\arabic{enumi}.\label{rule:All}] $\phi \limp \psi \Rightarrow \phi \limp \forall x \psi$ for $x \not\in \FV(\phi)$
    \item[\stepcounter{enumi}\refstepcounter{enumi}\LPT\arabic{enumi}.\label{rule:Ex}] $\phi \limp \psi \Rightarrow \exists x \phi \limp \psi$ for $x \not\in \FV(\psi)$
    \item[\stepcounter{enumi}\refstepcounter{enumi}\LPT\arabic{enumi}.\label{ax:eqRef}] $\forall x (x \= x)$
    \item[\stepcounter{enumi}\refstepcounter{enumi}\LPT\arabic{enumi}.\label{ax:eqTrans}] $\forall x y z (x \= y \land y \= z \limp x \= z)$
    \item[\stepcounter{enumi}\refstepcounter{enumi}\LPT\arabic{enumi}.\label{ax:eqRel}] $\forall \vec{x} \vec{y} (R(\vec{x}) \land \vec{x} \= \vec{y} \limp R(\vec{y}))$
    \item[\stepcounter{enumi}\refstepcounter{enumi}\LPT\arabic{enumi}.\label{ax:StFunc}] $f(\vec{\tau})\den \limp \tau_i\den$
    \item[\stepcounter{enumi}\refstepcounter{enumi}\LPT\arabic{enumi}.\label{ax:StVar}] $x\den$
    \end{enumerate}
    \vspace{0.5\lineskip}
  \end{minipage}
  \begin{minipage}[t]{0.52\linewidth}\renewcommand{\=}{\ensuremath{{\,=\!\!\!\!=\,}}}
    \vspace{0.5\lineskip}
    \begin{enumerate}
    \item[\stepcounter{enumi}\refstepcounter{enumi}\LPT\arabic{enumi}.\label{rule:MP}] $\phi, \phi \limp \psi \Rightarrow \psi$
    \item[\stepcounter{enumi}\refstepcounter{enumi}\LPT\arabic{enumi}.\label{ax:Con}] $\phi \land \psi \limp \phi, \phi \land \psi \limp \psi$
    \item[\stepcounter{enumi}\refstepcounter{enumi}\LPT\arabic{enumi}.\label{ax:Dis}] $\phi \limp \phi \lor \psi, \psi \limp \phi \lor \psi$
    \item[\stepcounter{enumi}\refstepcounter{enumi}\LPT\arabic{enumi}.\label{rule:Curry}] $\phi \land \psi \limp \chi \Rightarrow \phi \limp (\psi \limp \chi)$
    \item[\stepcounter{enumi}\refstepcounter{enumi}\LPT\arabic{enumi}.\label{ax:EFQ}] $\falsum \limp \phi$
    \item[\stepcounter{enumi}\refstepcounter{enumi}\LPT\arabic{enumi}.\label{ax:All}] $\forall x \phi \land \tau\den \limp \phi(x/\tau)$
    \item[\stepcounter{enumi}\refstepcounter{enumi}\LPT\arabic{enumi}.\label{ax:Ex}] $\phi(x/\tau) \land \tau\den \limp \exists x \phi$
    \item[\stepcounter{enumi}\refstepcounter{enumi}\LPT\arabic{enumi}.\label{ax:eqSym}] $\forall x y (x \= y \limp y\= x)$
    \item[\stepcounter{enumi}\refstepcounter{enumi}\LPT\arabic{enumi}.\label{ax:eqFunc}] $\forall \vec{x} \vec{y} (\vec{x} \= \vec{y} \land f(\vec{x})\den \limp f(\vec{x}) \= f(\vec{y}))$
    \item[\stepcounter{enumi}\refstepcounter{enumi}\LPT\arabic{enumi}.\label{ax:StC}] $c\den$
    \item[\stepcounter{enumi}\refstepcounter{enumi}\LPT\arabic{enumi}.\label{ax:StRel}] $R(\vec{\tau}) \limp \tau_i\den$\setcounter{memory}{\value{enumi}}
    \end{enumerate}
    \vspace{0,5\lineskip}
  \end{minipage}
  Rules {\color{darkred}$\LPT$}\ref{rule:All} and {\color{darkred}$\LPT$}\ref{rule:Ex} are called (universal and existential) generalisation, axioms {\color{darkred}$\LPT$}\ref{ax:All} and {\color{darkred}$\LPT$}\ref{ax:Ex} (universal and existential) instantiation and {\color{darkred}$\LPT$}\ref{ax:StC}, {\color{darkred}$\LPT$}\ref{ax:StFunc}, {\color{darkred}$\LPT$}\ref{ax:StRel} and {\color{darkred}$\LPT$}\ref{ax:StVar} are the strictness axioms. An instance of {\color{darkred}$\LPT$}\ref{ax:StRel} is $\tau \= \sigma \limp {\tau\den} \land \sigma\den$.
  Axiom {\color{darkred}$\LPT$}\ref{ax:StVar} is not explicitly included in the axiomatisation in \cite{vandenBerg_vanSlooten:2018} but is consistent with it, being derivable from the the axioms of $\HAP$.
  Since the axiom is present in Beeson's axiomatisation in \cite{Beeson:1985} and the axiomatisation of Troelstra's \cite{Troelstra:1998}, as well as \cite{Troelstra_vanDalen:1988a} (there called E${}^+$-logic), we include it for simplicity.
  This also makes sense for $\alpha$-equivalence, which is built into this system by construction; given {\color{darkred}$\LPT$}\ref{ax:StVar} one can also derive it from (more carefully formulated) {\color{darkred}$\LPT$}\ref{ax:All} and {\color{darkred}$\LPT$}\ref{ax:Ex}.
\end{Def}

\begin{Def}\label{def:prf}
  If $\Gamma$ is a set of formulae and $\phi$ a formula then $\Gamma \prf[\mathscr{D}] \phi$ means $\mathscr{D}$ is an $\LPT$-derivation (sequence of formulae) of $\phi$ using $\Gamma$ as additional axioms.
  We write $\Gamma \prf \phi$ to express that $\Gamma \prf[\mathscr{D}] \phi$ for some \( \LPT \)-derivation $\mathscr{D}$.
  If $\Gamma$ is empty we write $\prf[\mathscr{D}] \phi$.
\end{Def}

The notion of derivation introduced above satisfies the deduction theorem. We can also define weak equality and obtain:

\begin{Lma}[`Leibniz Lemma' of \cite{vandenBerg_vanSlooten:2018}\label{lma:Leibniz}]
  Defining $\sigma \simeq \tau$ as $\sigma\den \lor\; \tau\den \limp \sigma \= \tau$ we have $\prf \vec{\sigma} \simeq \vec{\tau} \limp \f{f}(\vec{\sigma}) \simeq \f{f}(\vec{\tau})$ and $\prf \sigma \simeq \tau \land \phi(x/\tau) \limp \phi(x/\sigma)$ for all function symbols $\f{f}$ and all formulae $\phi$ of the language.
\end{Lma}

\begin{Def}[$\HAP$\label{def:HAP}]
  Heyting Arithmetic with partial terms, $\HAP$, is the $\LPT$-theory in the language $\LHAP$. It is axiomatised by the following formulae for all variables $x$, $y$ and $z$ and formulae $\varphi \in \LHAP$. Recall, juxtaposition of terms abbreviates use of $\app$ (read as application).
  \begin{enumerate}
  \item[\refstepcounter{enumi}\HAP\arabic{enumi}.\label{HAP:arden}] $\fS(x)\den$, $(x + y)\den$, $(x \times y)\den$
  \item[\stepcounter{enumi}\HAP\arabic{enumi}.\label{HAP:arS}] $\fS(x) \= \fS(y) \limp x \= y$, $\0 \nonequality \fS(x)$
  \item[\stepcounter{enumi}\HAP\arabic{enumi}.\label{HAP:arA}] $x + \0 \= x$, $x + \fS(y) \= \fS(x + y)$
  \item[\stepcounter{enumi}\HAP\arabic{enumi}.\label{HAP:arM}] $x \times \0 \= \0$, $x \times \fS(y) \= (x \times y) + y$
  \item[\refstepcounter{enumi}\HAP\arabic{enumi}.\label{HAP:co}] $\cok x y \= x$, $(\cos x y)\den$, $\cos x y z \simeq x z (y z)$
  \item[\refstepcounter{enumi}\HAP\arabic{enumi}.\label{HAP:proj}] $\copl x \den$, $\copr x \den$, $\copl (\cop x y) \= x$, $\copr (\cop x y) \= y$, $\cop (\copl x) (\copr x) \= x$
  \item[\refstepcounter{enumi}\HAP\arabic{enumi}.\label{HAP:rec}] $\cosucc x \= \fS(x)$, $\cor x y \0 \= x$, $\cor x y \fS(z) \= y z (\cor x y z)$
  \item[\refstepcounter{enumi}\HAP\arabic{enumi}.\label{HAP:ind}] $\varphi(x/\0) \limp (\forall x (\varphi \limp \varphi(x/\fS(x))) \limp \forall x \varphi)$
  \end{enumerate}
  Terms and formulae from the fragment $\0,\fS,+,\times$ will be called arithmetical. If $\L$ is a language extending $\LHAP$ we will extend $\HAP$ to an $\L$-theory $\HAP(\L)$ by letting the induction schema {\color{darkred}$\HAP$}\ref{HAP:ind} apply to all formulae of $\L$.
\end{Def}

Not surprisingly $\HA$ directly embeds into $\HAP$. The next lemma summarises the key observations for this result.

\begin{Lma}\label{lma:HA->HAP}
  Let $\L_1 \subseteq \L_2$ be arbitrary first-order languages and $\Gamma_1$, $\Gamma_2$ sets of formulae in the respective language.
  Let $\prf_{\mathrm{I}}$ denote derivability in ordinary intuitionistic logic, as captured by the Hilbert-style system $\LPT$ plus the extra logical axiom $\tau \den$ for all terms $\tau$.
  Suppose $\Gamma_2 \prf \sigma \den$ for all terms $\sigma$ of $\L_1$. If $\Gamma_2 \prf \gamma$ for all $\gamma \in \Gamma_1$ and $\Gamma_1 \prf_{\mathrm{I}} \varphi$ then $\Gamma_2 \prf \varphi$, for all $\varphi \in \L_1$.
\end{Lma}
\begin{proof}
  Induction on derivations $\Gamma_1 \prf_{\mathrm{I}} \varphi$.
\end{proof}

\begin{Cor}\label{cor:HAP|-HA}
  $\HAP$ proves all theorems of $\HA$.
\end{Cor}

\subsection[Combinatory algebra in $\HAP$]{Combinatory algebra in $\bHAP$}
\label{sec:coHAP}

This subsection aims to describe how the basic facts of combinatory algebra are derivable in $\HAP$. In this, we follow \cite[VI.2.]{Beeson:1985}. The basic idea is to regard every object of the theory both as a number and as a partial function, where $\app$ denotes function application to an argument. It is in line with this reading we will write $\tau \vec{\sigma}$ and $\tau \app \vec{\sigma}$ for $\tau \app \sigma_0 \app \dotsb \app \sigma_n$, which in turn is a shorthand for $((\dotsb(\tau \app \sigma_0) \app \dotsb ) \app \sigma_n)$; thus association of $\app$ is to the left.

It is well known that combinatory algebra is Turing-complete, and thus defines the same class of functions as the recursive functions or the $\lambda$-calculus. Adapted to present purposes we observe that we can define a denoting `$\lambda$-term' for any partial term in the language, such that if applied to an argument they are weakly equal.

\begin{Prp}[$\lambda$-terms, Proposition 2.3 of \cite{vandenBerg_vanSlooten:2018}\label{prp:lambda}]
  For each term $\tau$ and variable $x$ in $\LHAP$, there is a term $\lambda x.\tau$ in the fragment without $\fS$, $+$ and $\times$ such that $\FV(\lambda x.\tau) = \FV(\tau) \setminus \{x\}$, $\HAP \prf \lambda x.\tau \den$ and $\HAP \prf \sigma \den \limp (\lambda x. \tau) \app \sigma \simeq \tau[x/\sigma]$.
\end{Prp}

As stated above we will call terms constructed in accordance with this lemma \emph{$\lambda$-terms}.
At first sight they might seem rather redundant, but note that a $\lambda$-term \emph{always} denotes, as opposed to the term it was constructed from (though the application of a $\lambda$-term to some argument need not denote, of course). Note that this also means that $\tau$ is \emph{not} a subterm of $\lambda x.\tau$ in general. Another consequence of the construction is that $\lambda$-terms do not contain any arithmetical function symbols, but only $\app$. We will use the usual notational conventions for stacking $\lambda$s, that is $\lambda x.\lambda y. \tau$ will be written $\lambda x y.\tau$ etc.

\begin{Prp}[Primitive recursion in $\HAP$\label{prp:PRinHAP}]
  For each function symbol of $\LPRA$ defined as a primitive recursive function there is a closed combinator term in the fragment of $\LHAP$ without $\fS$, $+$ and $\times$ such that $\HAP$ proves its defining equations. More precisely there is an interpretation $\tau$ of these function symbols of $\LPRA$ into the terms of $\LHAP$ such that for all such $\f{f}$, $\f{g}$, and $\f{h}_{(j)}$ with arities $n$, $k + 2$ and $k$
  \begin{align*}
    \HAP &\vdash \forall x\; \tau_{\fZ} x \= \f0\text,\\
    \HAP &\vdash \forall x\; \tau_{\fS} x \= \fS(x)\text,\\
    \HAP &\vdash \forall \vec{x}\; \tau_{\f{\pi}^k_i} \vec{x} \= x_i\text,\\
    \HAP &\vdash \forall \vec{x}\; \tau_{\circ^k_n(\f{f},\f{h}_0,\dotsc,\f{h}_{n - 1})} \vec{x} \= \tau_{\f{f}} (\tau_{\f{h}_0} \vec{x}) \dotsb (\tau_{\f{h}_{n - 1}} \vec{x})\text,\\
    \HAP &\vdash \forall \vec{x}\; \tau_{\rec^k(\f{g},\f{h})} \vec{x}\, \f0 \= \tau_{\f{h}} \vec{x}\text,\\
    \HAP &\vdash \forall \vec{x}\, y\; \tau_{\rec^k(\f{g},\f{h})} \vec{x} (\fS(y)) \= \tau_{\f{g}} \vec{x} y (\tau_{\rec^k(\f{g},\f{h})} \vec{x} y)\text,
    \intertext{(where the length of $\vec{x}$ is $k$) and consequently}
    \HAP &\vdash \tau_{\f{f}} \num{a}_0 \dotsb \num{a}_{n - 1} \= \num{\val(\f{f},\vec{a})}
  \end{align*}
  for every $a_0,\dotsc,a_{n - 1} \in \N$. In particular $\HAP \vdash \tau_{\f{f}}\den$ for every such $\f{f}$.
  \begin{proof}
    The proof is standard. See e.g.\ \cite[Theorem 1.2.4]{vanSlooten:2014} and \cite[Theorem 9.3.11]{Troelstra_vanDalen:1988b}.
  \end{proof}
\end{Prp}

Fixing $\LHA$ we can extend this interpretation compositionally to an interpretation $\mathcal{T}$ of $\LPRA$ into $\LHAP$, yielding the following result.

\begin{Prp}\label{prp:PRHA->HAP}
  Recall $\PRHA$ is Heyting Arithmetic axiomatised in $\LPRA$. For every formula $\phi$ of $\LPRA$,
  \begin{align*}
    \PRHA \prf \phi \Rightarrow \HAP \prf \mathcal{T}(\phi)
  \end{align*}
  In particular, if $\phi \in \LPRA$ is a true (in $\N$) $\Sigma_1$-sentence, then $\HAP \prf \mathcal{T}(\phi)$.
  \end{Prp}
  \begin{proof}
    Apply Lemma \ref{lma:HA->HAP} to $\mathcal{T}[\PRHA] \subset \HAP$.
  \end{proof}

\begin{Prp}[Recursion in HAP\label{prp:coHAP}]
  Combinatory algebra in $\HAP$ encompasses all recursive functions: there are (partial) terms $\cod$, $\cofix$ and $\comin$ such that
  \begin{align*}
    \HAP &\prf x \= y \limp \cod x y f g z \simeq f z\text,\\
    \HAP &\prf x \nonequality y \limp \cod x y f g z \simeq g z\text,\\
    \HAP &\prf \cofix f \den \land\; \cofix f x \simeq f (\cofix f) x\text,\\
    \HAP &\prf \comin f x \den \liff \exists y (f x y \= 0 \land \forall z < y (f x z \den))
  \end{align*}
  and
  \begin{align*}
    \HAP &\prf \comin f x \den \limp (f x (\comin f x) \= 0 \land \forall y < (\comin f x)(f x y > 0))\text.
  \end{align*}
\end{Prp}
\begin{proof}
  Again standard. See \cite[Proposition 1.2.1.]{vanSlooten:2014}, \cite[VI.2.6--8]{Beeson:1985} and \cite[9.3.7--11]{Troelstra_vanDalen:1988b}.
\end{proof}

The above thus embeds $\LPRA$ (and more) into $\LHAP$. It will, however, be of vital importance that $\LHA \subset \LPRA$ is embedded in $\LHAP$ in a way that does not involve $\app$ and the combinators; recall that $\LHA \subset \LHAP$ \emph{by definition} as the arithmetical formulae of $\LHAP$, which are not affected by $\mathcal{T}$.

By interpreting $\app$ as Kleene application, we also have a result in the converse direction.

\begin{Prp}\label{prp:HAP->HA}
  There is an interpretation $\mathcal{F}$ of $\LHAP$ in $\LHA$ which interprets $\app$ as Kleene application, such that
  \begin{enumerate}
  \item $\HA \prf \phi \liff \mathcal{F}(\phi)$ for all $\phi \in \LHA$,\label{HA->HAP:1}
  \item $\HAP \prf \phi \Rightarrow \HA \prf \mathcal{F}(\phi)$ for all $\phi \in \LHAP$.
  \end{enumerate}
\end{Prp}

\begin{Cor}[Proposition 2.4 of \cite{vandenBerg_vanSlooten:2018}]\label{cor:HAP->HA}
  $\HAP$ is conservative over $\HA$.
\end{Cor}

\subsection[Realizability in $\HAP$]{Realizability in $\bHAP$}
\label{sec:HAPr}

Realizability in $\HAP$ is one of the main tools of \cite{vandenBerg_vanSlooten:2018}, and we will utilise many of their results, and those of similar systems like in \cite{Beeson:1985,Troelstra_vanDalen:1988b}. We present this technical background, mainly without proof, in this subsection, together with some additional machinery.

We will, in fact, extend these definitions of realizability to languages $\L$ extending $\LHAP$ with an unspecified set of relation symbols and parameterise this extension by the clauses for these symbols. We let $\L$ be fixed throughout this subsection.

\begin{Def}[Realizability]\label{def:real}
  For each term $\tau$ and formula $\varphi$ we define the formula $\tau \rea \varphi$ after renaming of bound variables in $\varphi$ to avoid conflicts with $\tau$ and nested bindings. The notion of realizability we will use is defined via the following clauses:
  \begin{align*}
    \tau \rea (\sigma_1 \= \sigma_2) &\text{ is } \sigma_1 \= \sigma_2\\
    \tau \rea \falsum &\text{ is } \falsum \\
    \tau \rea (\phi \land \psi) &\text{ is } ( \copl \tau \rea \phi ) \land ( \copr \tau \rea \psi )\\
    \tau \rea (\phi \lor \psi) &\text{ is } (\copl \tau \= \0 \limp \copr \tau \rea \phi) \land (\copl \tau \nonequality \0 \limp \copr \tau \rea \psi) \\
    \tau \rea (\phi \limp \psi) &\text{ is } \forall y (y \rea \phi \limp ({\tau y \den} \land \tau y \rea \psi)) \\
    \tau \rea \forall x \phi &\text{ is } \forall x ({\tau x \den} \land \tau x \rea \phi) \\
    \tau \rea \exists x \phi &\text{ is } \copr \tau \rea (\phi(x/\copl \tau)) \\
    \tau \rea R(\sigma_0,\dotsc,\sigma_{n - 1}) &\text{ is } Q_R(\sigma_0,\dotsc,\sigma_{n - 1},\tau)
  \end{align*}
  where $y$ is a fresh variable (which can be chosen in a systematic fashion which we will not specify). Here $Q_R$ is an as yet unspecified relation symbol whose arity is one greater than that of $R$.
\end{Def}

That this is a well-defined notion, in that it respects $\alpha$-equivalence, can be proved by structural induction.

\begin{Lma}[Partly Lemma 4.4.6.(i) of \cite{Troelstra_vanDalen:1988a}\label{lma:reaFV}]
  For every formula $\phi$ and term $\tau$, $\tau \rea \phi$ is negative and the free variables of $\phi$ are among the free variables of $\tau \rea \phi$, which are among the free variables of $\phi$ and $\tau$.
\end{Lma}

For the rest of this section we will simply let $\rea R$ denote the unspecified relation $Q_R$ in the definition of $\rea$ above, whence the final clause reads
\begin{align*}
  \tau \rea R(\sigma_0,\dotsc,\sigma_{n - 1}) &\text{ is } (\rea R)(\sigma_0,\dotsc,\sigma_{n - 1},\tau)
\end{align*}
for relations $R$. We will make some concrete choices for this $\rea R$ in particular cases in section \ref{sec:rIID}. This draws attention to a potential subtlety, in that $\tau \rea \varphi(\vec{x})$ could a priori mean both $\tau \rea (\varphi(\vec{x}))$ and $(\tau \rea \varphi)(\vec{x})$. As our choice of notation for the base case indicates, however, these are the same as long as $\tau$ contains none of the $\vec{x}$.

\begin{Lma}[Lemma 4.4.6.(ii) of \cite{Troelstra_vanDalen:1988a}\label{lma:readist}]
  Substitution distributes over realizability: Let $\varphi$ be a formula and $\tau$ a term. For all equally long finite sequences $\vec{u}$ and $\vec{\sigma}$ of variables and terms, respectively, $(\tau \rea \varphi)(\vec{u}/\vec{\sigma}) = (\tau[\vec{u}/\vec{\sigma}]) \rea\, (\varphi(\vec{u}/\vec{\sigma}))$. In particular, $(x \rea \varphi)(x/\tau) = \tau \rea \varphi$ if $x$ is fresh.
\end{Lma}

To disambiguate when there is a potential clash we will let $\tau \rea \varphi(\vec{y})$ mean $\tau \rea (\varphi(\vec{y}))$ etc., since this seems to be the natural reading when $\tau$ is a(n explicitly given) single variable. 

\begin{Lma}\label{lma:reasubst}
  Let $n \in \N$, $R$ be an $n$-ary relation symbol, $\tau$ be a term and $\vartheta$ be a formula without free occurrences of $\fv_n$. Then $\tau \rea (\varphi(R/\lambda \vec{\fv}_{< n}.\vartheta)) = (\tau \rea \varphi)(\rea R/\lambda \vec{\fv}_{\leq n}.\fv_n \rea \vartheta)$ for every $\varphi \in \L$.
\end{Lma}
\begin{proof}
  Induction on the construction of $\vartheta$.
\end{proof}

\begin{Lma}\label{lma:int->rea}
  Realizability is closed under derivability in $\LPT$\textup{:} For every derivation $\mathscr{D}$ with $\Gamma \prf[\mathscr{D}] \varphi$ for formulae $\Gamma$ and $\varphi$ of $\L$ there is a term $\tau$ with at most $\fv_0$ free such that $\HAP(\L) \prf \forall x ( x \rea \Gamma \limp \tau(x)\den \land \tau(x) \rea \varphi)$.
\end{Lma}
\begin{proof}
  By induction on the depth of $\mathscr{D}$. The proof is mostly as in \cite[VII.\ Theorem 1.6]{Beeson:1985}.
\end{proof}

While we rely heavily on the work in \cite{vandenBerg_vanSlooten:2018}, our definitions of realizability are not literally the same, since there disjunction is a defined predicate ($\varphi \lor \psi$ is $\exists n ((n \= \0 \limp \varphi) \land (n \nonequality \0 \limp \psi))$). This difference is immaterial, however, since the notions are equivalent in the following sense. In the following lemma let $\reha$ denote the notion of realizability from \cite{vandenBerg_vanSlooten:2018}.

\begin{Lma}\label{lma:eqrea}
  The realizability notion $\rea$ is equivalent to $\reha$ in the following sense: for every formula $\varphi \in \LHAP$ there are total terms $\mu$ and $\nu$ such that $\HAP \vdash \forall x ( x \rea \varphi \limp \mu x \reha \varphi)$ and $\HAP \vdash \forall x (x \reha \varphi \limp \nu x \rea \varphi) $. Consequently $\HAP \vdash \exists x (x \rea \varphi) \liff \exists x (x \reha \varphi)$.
\end{Lma}

Since we can thus identify the two notions of realizability we get the following as a consequence of Corollary 3.6 from \cite{vandenBerg_vanSlooten:2018}:

\begin{Prp}\label{prp:HAPe->HA}
  $\HAP$ plus the schema $\varphi \liff \exists x (x \rea \varphi)$ for $\varphi \in \LHA$ is conservative over $\HA$. In particular, for every formula \( \varphi \in \LHA \), if \( \HAP \vdash \exists x (x \rea \varphi) \) then \( \HA \vdash \varphi \).
\end{Prp}

It is crucial for the above proposition that $\LHA$ is a subset of $\LHAP$, instead of being embedded into it in a way involving the combinators or the application function symbol specific to the language \( \LHAP \), like the inclusion of the rest of $\LPRA$ in Proposition \ref{prp:PRinHAP}.

% -------------------------------
% -------------------------------
\section{A hierarchy of almost negative formulae}
\label{sec:Lambda}
% -------------------------------
% -------------------------------

An important subtheory of $\IID1$ is that given by the fixpoint axioms for operator forms which are almost negative. Recall a formula is almost negative if it contains no disjunctions and the matrix of any existential quantifier is an equation between terms. \( \AnHier \) denotes the class of almost negative formulae.

In this section we introduce a hierarchy $(\AnHierC_n)$ of formulae which exhausts the almost negative formulae $\AnHierC$. In classical arithmetic, the arithmetic (quantifier) hierarchy classifies formulae of the language according by quantifier complexity (that is, quantifier alternations). This hinges on the existence of a prenex normal form, and in particular on De Morgan's laws and the possibility to define $\limp$ in terms of $\lor$ and $\neg$. For intuitionistic arithmetic this is no longer a possibility and the situation is not as simple. A few hierarchies for intuitionistic arithmetic have been proposed, in this context we should mention Wolfgang Burr whose $\Phi^{\prime}$-hierarchy in \cite{Burr:2000} is very close to the one we consider (see also Burr's \cite{Burr:2004}).

This formula hierarchy will behave relative to $\AnHierC$ in several respects similarly to how the arithmetic hierarchy behaves relative to the entire language in the classical case. In addition to exhausting the almost negative formulae (in a sense made precise below), each level of the hierarchy contains a partial satisfaction predicate for that level: for each $n \in \N$ there is a $\AnHierC_n$-formula $\fdSat_n$ which $\HAP$ proves to be a (compositional) satisfaction predicate for formulae in $\AnHierC_n$. The construction of the $\fdSat_n$ will be the main content of section \ref{sec:SatHAP}.

\subsection{Stratifying the almost negative formulae}
\label{sec:stratify}

For now we turn to more elementary properties of $\AnHierC_n$. Two of the most notable of these are the Diagonal Lemma \ref{lma:Diag} and the closure of $\AnHierC_n$ under substitution into strictly positive positions (Lemma \ref{lma:PosInsert}).
These properties play crucial roles in the interpretability argument of Theorem \ref{thm:IIDP->HAP}.

\begin{Def}
  $\Sigge$, the set of $\Sigma$-equations, is the set of formulae of the form $\exists x \, \sigma = \tau$, where $\sigma$ and $\tau$ are terms. $\AnHierC_0$ is the smallest set closed under conjunction which contains $\Sigge$ and the atomic formulae. $\AnHierC_{n + 1}$ is the smallest set that contains $\Sigge$ and atomic formulae which is closed under conjunction, universal quantification and implications with $\AnHierC_n$ antecedents; that is, if $\varphi \in \AnHierC_{n + 1}$ and $\psi \in \AnHierC_n$ then $\psi \limp \varphi \in \AnHierC_{n + 1}$.

  The normal form hierarchy $ ( \AnHierN_n )_n$ is defined as follows: $\AnHierN_0$ is the set conjunctions $ \psi \land \bigland_{i=1}^k \varphi_i $ where \( \psi \in \Sigge \), \( k \ge 0 \) and each \( \varphi_i \) for $i \leq k$ is an atomic formula other than $\falsum$ or an equation (whence $\AnHierN_0 = \Sigge$ for $\LHA$, $\LPRA$ and $\LHAP$).
  $\AnHierN_{n + 1}$ is the set of universally quantified conjunctions  of implications with \( \AnHierN_n \) antecedents and \( \AnHierN_0 \) consequents, namely formulae of the form $\forall x \bigland_{i=0}^{k} ( \phi_i \limp \psi_i )$ where \( \{ \phi_0 , \dotsc, \phi_k \} \subseteq \AnHierN_n \) and \( \{ \psi_0 , \dotsc, \psi_k \} \subseteq \AnHierN_0 \).
  Let $\AnHierN = \bigcup_{n \in \N} \AnHierN_n$.
\end{Def}

It should be noted that this stratification of $\AnHier$ is closed under $\alpha$-equivalence and thus is well defined.
A consequence of the definition is that the $\AnHierC_n$ form a strictly telescoping hierarchy:

\begin{Lma}\label{lma:Lambdatelescope}
  For every $n \in \N$, $\AnHierC_n \subset \AnHierC_{n + 1}$.
\end{Lma}

The hierarchy \( ( \AnHierN_n )_n \) represents normal forms for formulae of \( \AnHier \) in the sense that \( \AnHierN _n \subseteq \AnHier_n \) for every \( n \), and every formula in \( \AnHier_n \) is provably equivalent to a formula in \( \AnHierN \) over the theory \( \HA \) or \( \HAP \) as appropriate. This latter result is strengthened in Lemma~\ref{lma:Lambdatransform}.

We are interested here in the almost negative formulae, which is why we allow $\Sigge$ as `atoms,' much as $\Delta_0$-formulae are treated as `atoms' in the classical arithmetical hierarchy.
In the latter case, however, we have generous closure conditions on this set of `atoms,' even though this makes no real difference further on in the arithmetical hierarchy, while here we are requiring the special `atoms' to have a very concrete form.
We could, for the purposes of this paper, have let $\Sigge$ be closed under (at least) conjunctions, disjunctions and existential quantification and subsequently used formulae constructed therefrom by conjunctions, implications and universal quantifications instead of the almost negative formulae (thus thereby essentially allowing general $\Sigma_1$-formulae as `atoms') without much extra work.
  Since the almost negative formulae seem to be the established standard in these contexts, however, we will not generalise in this direction. This will have the added advantage of making the definitions of the satisfaction predicates somewhat clearer since no connectives or quantifiers are in the scope of an existential quantifier.

In the other direction, as opposed to Heyting arithmetic in the language of $\PRA$, every (partial) recursive function is represented by a term in $\HAP$. Thus, adding an existential quantifier to form a $\Sigma$-equation does not yield increased expressibility.
So we could have simplified the definition of $\AnHierN$ in this setting by leaving out the existential quantifiers altogether. However, we gain little by doing so, and thus prefer to have $\AnHierN$ formulated as a template for all adequate languages.

The essential difference between the $\AnHier$-hierarchy and Burr's $\Phi^{\prime}$ from \cite{Burr:2000} is that $\AnHierC_n$ is also closed under conjunction to exhaust the almost negative formulae. Both can be seen as mirroring the ordinary arithmetic hierarchy, in that they are classically equivalent to it.

In the following lemma, let $\PAP$ denote the result of adding the law of excluded middle to $\HAP$.\footnote{To the best of our knowledge the theory $\PAP$ has not been studied. We thus make no claims on suitability for any particular purpose, but merely use it as a name for adding $\TND$ to $\HAP$.} Given a set \( \Gamma \) of formulae and a theory \( \mathrm T \), recall \( \Gamma(\mathrm T) \) denotes the set of formulae that are provably in \( \mathrm T \) equivalent to some formula in \( \Gamma \).

\begin{Lma}\label{lma:Lambdarel}
  $\AnHier_n$ relates to almost negative formulae and the usual arithmetical hierarchy in the following way:
  \begin{enumerate}
  \item The $\AnHierC_n$ exhaust the almost negative formulae, that is $\AnHierC = \bigcup_{n \in \N} \AnHierC_n$.
  \item In the arithmetic languages $\LPRA$ and $\LHAP$\textup{:} Let $\HA^*$ be $\PRHA$ or $\HAP$. Then $\AnHierN_0(\HA^*)$, the formulae equivalent to $\AnHierN_0$-formulae over $\HA^*$, is precisely the formulae equivalent to $\Sigge$-formulae over $\HA^*$, as well as the formulae equivalent to $\Sigma_1$-formulae over $\HA^*$.
  \item In the arithmetic languages $\LPRA$ and $\LHAP$\textup{:} Let $\PA^*$ be $\PRPA$ or $\PAP$. For $n > 0$, $\AnHierN_n(\PA^*) = \Pi_{n + 1}(\PA^*)$. Moreover, $\AnHierN_0(\PA^*) = \Sigge(\PA^*) = \Sigma_1(\PA^*)$.
  \end{enumerate}
  \begin{proof}
    Observe that $\AnHierC_n \subseteq \AnHierC$ for all $n \in \N$.
    We show that \( \AnHierC \subseteq \bigcup_n \AnHier_n \) by induction on the complexity of almost negative formulae:
    \begin{itemize}
    \item If $\varphi$ is an atom or a $\Sigge$-formula, then $\varphi \in \AnHierC_0$.
    \item The induction steps for $\land$ and $\forall$ are straightforward.
    \item If $\varphi = (\vartheta \limp \psi)$ for $\vartheta,\psi \in \bigcup_{k \in \N} \AnHierC_k \cap \AnHierC$, let $n$ be such that both are in $\AnHierC_n$ (by Lemma \ref{lma:Lambdatelescope}). Then $\varphi \in \AnHierC_{n + 1}$.
    \end{itemize}
    This concludes the proof of the first statement.

    For the remaining statements, clearly $\Sigge = \AnHierN_0 \subseteq \Sigma_1$ by definition. We show by induction on $n > 0$ that $\AnHierN_n \subseteq \Pi_{n + 1}(\PA^*)$.
    \begin{itemize}
    \item Let $\varphi = \forall x \bigland_i (\vartheta_i \limp \psi_i)$, where $\vartheta_i,\psi_i \in \AnHierN_0 \subseteq \Sigma_1$. Then by classical logic $\varphi \in \Pi_2(\PA^*)$.
    \item Let $k \geq 1$ and $\varphi = \forall x \bigland_i (\vartheta_i \limp \psi_i)$, where $\vartheta_i \in \AnHierN_k \subseteq \Pi_{k + 1}(\PA^*)$ and $\psi_i \in \AnHierN_0 \subseteq \Sigma_1$. Then by classical logic $\varphi \in \Pi_{k + 2}(\PA^*)$.
    \end{itemize}
    Conversely, since $\Delta_0$-formulae in $\LPRA$ and $\LHAP$ are equivalent to atoms in $\HA^*$, $\Sigma_1 \subseteq \Sigge(\HA^*) = \AnHierN_0(\HA^*) \subseteq \AnHierN_0(\PA^*)$ by structural induction. Hence $\Sigma_1(\HA^*) = \AnHierN_0(\HA^*)$ and the corresponding for $\PA^*$. We show that $\Pi_{n + 1}(\PA^*) = \AnHierN_n(\PA^*)$ by induction on $n > 0$.
    \begin{itemize}
    \item Suppose $\varphi = \forall x \vartheta$, where $\vartheta \in \Sigma_1 \subseteq \AnHierN_0(\PA^*)$. Then $\varphi \in \AnHierN_1(\PA^*)$. Thus $\Pi_2 \subseteq \AnHierN_1(\PA^*)$, whence $\Pi_2(\PA^*) = \AnHierN_1(\PA^*)$ by above.
    \item Suppose $\varphi = \forall x \vartheta$, where $\vartheta \in \Sigma_{k + 1}$. Let $\tilde\neg \chi$ be $\chi \limp \exists x (\0 \= \fS(\0))$ for formulae $\chi$. Then $\tilde\neg \vartheta \in \Pi_{k + 1}(\PA^*) = \AnHierN_k(\PA^*)$ by induction hypothesis. Let $\psi \in \AnHierN_k$ be equivalent. Then $\PA^* \prf \varphi \liff \forall x \tilde\neg \psi$, where $\forall x \tilde\neg \psi \in \AnHierN_{k + 1}$. Note that this would not work with $\neg$, since $\falsum$ is not a $\AnHierN_0$-formula. So $\Pi_{k + 2} \subseteq \AnHierN_{k + 1}(\PA^*)$ and consequently $\Pi_{k + 2}(\PA^*) = \AnHierN_{k + 1}(\PA^*)$.\qedhere
    \end{itemize}
  \end{proof}
\end{Lma}

In particular, $\Sigge(\HAP)$ is closed under conjunction, disjunction, existential quantification and bounded universal quantification.

It should be intuitively clear that $\AnHierC_n$ is $\HAP$-equivalent to $\AnHierN_n$. In order to define a satisfaction predicate for $\AnHierC_n$, however, we need to have access to an explicit transformation from $\AnHierC_n$ to its normal form $\AnHierN_n$ in $\HAP$. Corresponding transformations can be defined for languages other than $\LHAP$. Since it is clear that $\AnHierN_n \subseteq \AnHierC_n$, this transformation yields the inclusions $\AnHierN_n \subseteq \AnHierC_n \subseteq \AnHierN_n(\HAP)$ and $\AnHierN_n \subseteq \AnHierC_n \subseteq \AnHierN_n(\PAP)$ for all $n$.

We now fix a standard coding of the language $\LHAP$ in $\HA$, and thus in $\HAP$. The features we will assume are that the set of codes of two types of syntactic objects never intersect, that no two distinct syntactic objects are assigned the same code, that $\0$ is not the code of any syntactic object and that the code of formulae and terms are (numerically) strictly larger than their (immediate) constituents. A similar note on the coding of sequences are that they are assumed to be strictly larger than their length and all of their elements. Thus we can reason about $\AnHier_n$-formulae in a theory of arithmetic.

\begin{Lma}\label{lma:Lambdatransform}
  There are primitive recursive functions $\lambda_n \colon \AnHierC_n \longrightarrow \AnHierN_n$ which transforms $\AnHierC_n$-formulae to \( \HAP \)-provably equivalent normal forms.
  \begin{proof}
    We omit the actual construction, from which we can see that $\HAP \prf \varphi \liff \lambda_n(\varphi)$ for $\varphi \in \AnHierC_n$ but which is otherwise unenlightening.
  \end{proof}
\end{Lma}

The $\lambda_n$ can be chosen to be (provably in $\HAP$) primitive recursive in the structure of their arguments.

We finally turn to the two key properties of the hierarchy.

\begin{Lma}
  \label{lma:PosInsert}
  If $\phi \in \AnHierC_n$ has only strictly positive occurrences of the $k$-ary relation symbol $R$ and $\theta \in \AnHierC_m$ then $\phi(R/\theta) \in \AnHierC_{\max(n,m)}$.
  \begin{proof}
    Let $M = \max(n,m)$. The proof is by induction on the complexity of $\phi$.
    \begin{itemize}
    \item Let $\sigma$ and $\tau$ be terms. Then $(\sigma \= \tau)(R/\theta) = (\sigma \= \tau) \in \AnHierC_M$.
    \item We have $(R(\vec{\tau}))(R/\theta) = \theta(\vec{\fv}_{< k}/\vec{\tau}) \in \AnHierC_M$.
    \item We have $(Q(\vec{\tau}))(R/\theta) = Q(\vec{\tau}) \in \AnHierC_M$.
    \item Let $\sigma$ and $\tau$ be terms. Then $(\exists x \sigma \= \tau)(R/\theta) = (\exists x \sigma \= \tau) \in \AnHierC_M$.
    \item Let $\chi, \psi \in \AnHierC_n$ have only strictly positive occurrences of $R$ and be such that $\chi(R/\theta), \psi(R/\theta) \in \AnHierC_M$. Then $(\chi \land \psi)(R/\theta) = (\chi(R/\theta) \land \psi(R/\theta)) \in \AnHierC_M$.
    \item In case $n \geq 1$, let $\psi \in \AnHierC_n$ have only strictly positive occurrences of $R$ and be such that $\psi(x/z)(R/\theta) \in \AnHierC_M$ for all $z$. Then $(\forall x \psi)(R/\theta) = \forall y (\psi(x/y)(R/\theta)) \in \AnHierC_M$ for some fresh $y$.
    \item In case $n \geq 1$, let $\psi \in \AnHierC_n$ have only strictly positive occurrences of $R$ and be such that $\psi(R/\theta) \in \AnHierC_M$, and $\chi \in \AnHierC_{n - 1}$ be such that $\chi \limp \psi$ has only strictly positive occurrences of $R$. Then $(\chi \limp \psi)(R/\theta) = (\chi \limp \psi(R/\theta)) \in \AnHierC_M$ since $R$ does not occur in $\chi$.
    \end{itemize}
    This concludes the proof.
  \end{proof}
\end{Lma}

Since $\HAP$ has terms for all primitive recursive functions, there is a term $\fdsub$ for the substitution function in $\LHAP$. This is used to prove the following:

\begin{Lma}[Diagonal Lemma in $\HAP$]\label{lma:Diag}
  $\AnHierC_n$ and $\AnHierN_n$ are closed under diagonalisation: If $\varphi \in \AnHierC_n$ (resp.~$\in \AnHierN_n$) has free variables $\fv_0, \dotsc, \fv_k$, then there is a formula $\psi \in \AnHierC_n$ (\( \psi \in \AnHierN_n \)) with free variables $\fv_0, \dotsc, \fv_{k - 1}$ and
  \begin{align*}
    \HAP \vdash \psi(\vec{x}) \liff \varphi(\vec{x},\GN{\psi})\text,
  \end{align*}
  where $\vec{x}$ has length $k$.
  \begin{proof}
    This is the standard diagonal argument.
  \end{proof}
\end{Lma}

% -------------------------------
% -------------------------------
\subsection{Satisfaction predicates for almost negative formulae}
\label{sec:SatHAP}
% -------------------------------
% -------------------------------

Given terms $\eta$, $\nu$ and $\tau$ we write $\eta_{\nu}^{\tau}$ for the term which is the `update' of the function $\eta$ with the new value $\tau$ for argument $\nu$.
In the notation of section~\ref{sec:HAP}, $\eta_{\nu}^{\tau}$ is the term $\lambda u. \cod u \nu (\cok \tau) \eta u$, where $u$ is fresh. (Recall $\cod$ is the `decision by case' term from Proposition~\ref{prp:coHAP}.)
This term is such that $\HAP \prf \forall x y (\eta_x^{\tau} \app x \simeq \tau \land (x \nonequality y \limp \eta_x^{\tau} \app y \simeq \eta \app y))$.
Note, there is no assumption that \( \tau \) denotes.
For sequences \( \vec \tau \) and \( \vec \nu \) of terms (of equal length) we define $e_{\vec{\nu}}^{\vec{\tau}}$ by iterating the above construction.

Since partial combinatory algebras are Turing complete (see Proposition \ref{prp:coHAP} and \cite[VI.\ Theorem 2.8.1]{Beeson:1985}) it is possible to define a valuation function for partial terms of $\LHAP$.

\begin{Def}[Valuation function in $\HAP$\label{def:val}]
  Using Proposition \ref{prp:PRinHAP} we have formulae and closed terms corresponding to the following primitive recursive relations (which will use the usual notation for applying to their arguments) and functions: $\fdSeq$, $\fdlh$, $\fdVar$, $\fdFV$, $\fdBV$ (the second argument is a (code of a) free/bound variable of the first), $\fdTerm$, %$\fdFmla$, $\fdSent$,
  $\fdsub$, %$\fdSigge$,
  $\fdAnHierN[n]$ and $\fdlambda_n$.% and $\fdnum$ (the formal numeral function).

  Let $\fdval$ be a partial term in $\HAP$ such that $\HAP \prf \fdval \app e \app \GN{\f{c}} \= \f{c}$ for every constant symbol $\f{c}$, $\HAP \prf \fdVar(x) \limp \fdval \app e \app x \simeq e \app x$, and $\HAP \prf \fdval \app e \app (\GN{\f{f}}(\vec{a})) \simeq \f{f}(\overrightarrow{\fdval \app e \app a})$ for every function symbol $\f{f}$ (recall the notational convention that $\GN{\alpha}$ means $\num{\GN{\alpha}}$, the numeral of the Gödel code of $\alpha$, for expressions $\alpha$).
\end{Def}

These defined relations will notationally be treated as predicate symbols.

Recall that in $\HAP$ every object in the domain is a partial function and, in particular, is a partial variable assignment mapping codes of variables to values.
Thus it is not necessary to place any restriction on the $e$ in $\fdval \app e \app z$ such as, for example, that \( e \) represents an assignment for the variables of (the term coded by the value of the variable) $z$.

\begin{Lma}\label{lma:val}
  The valuation function $\fdval$ acts as a partial valuation function for partial terms of $\LHAP$. In particular:
  \begin{enumerate}
  \item $\HAP \prf \forall e, f, t ((\fdTerm(t) \land \forall x < t (\fdFV(t,x) \limp e \app x \simeq f \app x)) \limp \fdval \app e \app t \simeq \fdval \app f \app t)$.
  \item $\HAP \prf \forall e, t, v, x (\fdTerm(t) \land \neg\fdFV(t,v) \limp \fdval \app e_v^x \app t \simeq \fdval \app e \app t)$.
  \item $\HAP \prf \forall e, s, t, x (\fdVar(x) \land \fdTerm(t) \land \fdTerm(s) \limp \fdval \app e \app (\fdsub \app t \app x \app s) \simeq \fdval \app e_x^{\fdval \app e \app s} \app t)$.
  \item $\HAP \prf \forall e (\fdval \app e \app \GN{\tau} \simeq \tau(e \app \GN{\fv_0} , \dotsc, e \app \GN{\fv_{n-1}}))$ for every partial term $\tau$ with free variables among $\fv_0,\dotsc,\fv_{n - 1}$.
  \end{enumerate}
  \begin{proof}
    The first three are proved by induction on $t$ in $\HAP$. The final claim is an induction on the construction of $\tau$.
  \end{proof}
\end{Lma}

In the interest of readability we will conflate coding and concatenation in the coded language. Any logical symbols occurring within square brackets $[\phantom{\exists}]$ refers to the codes of these symbols, concatenated with the other (named) constructs appearing in the brackets. For instance, $[\exists v (s \= t)]$ below is the code $\GN{\exists \fv_0 (\fv_1 \= \fv_2)}$ with the codes $v$, $s$, $t$ formally substituted for $\GN{\fv_0}$, $\GN{\fv_1}$, $\GN{\fv_2}$ respectively.

\begin{Thm}[Partial Satisfaction Predicates in $\HAP$\label{thm:Sat}]
  For every $n$ there is a $\AnHierN_n$-formula $\fdSat_n(e,F)$ which $\HAP$ proves is a compositional satisfaction predicate for $\AnHierC_n$-formulae.
  \begin{proof}
    The proof is cumbersome but straightforward. Essentially, we mimic standard constructions of $\Sigma_n$-satisfaction predicates in $\PA$, as in e.g.\ Hájek and Pudlák's \cite{Hajek_Pudlak:1993}.

    We first construct an intermediate satisfaction predicate for $\AnHierN_n$-formulae in $\LHAP$ by induction on $n$.
    \begin{itemize}
    \item First $\fdSat_0^{\prime\prime}(e,F)$ is
      \begin{align*}
        \exists x,v,s,t(\fdTerm(s) \land \fdTerm(t) \land \fdVar(v) \land F \= [\exists v (s = t)] \land \fdval \app e_v^x \app s\= \fdval \app e_v^x \app t)
      \end{align*}
      This is straightforwardly $\HAP$-equivalent to a $\Sigge$-formula. The predicate $\fdSat^{\prime}_0$ is thus the $\AnHierN_0$ normal form of this formula, which is $\HAP$-equivalent to it by Lemma \ref{lma:Lambdatransform}.
    \item We define $\fdSat^{\prime\prime}_{n + 1}(e,F)$ as
      \begin{align*}
        \fdAnHierN[n + 1](F) &\land \forall x,v,G,s (\fdVar(v) \land F \= [\forall v G]\\*
        &\qquad \land \fdSeq(s) \land \fdlh \app s > \0 \land G \= \fdconjs \app s\\*
        &\quad \limp \forall i < \fdlh \app s\, \forall f,g < s_i
        (s_i \= [g \limp f] \land \fdAnHierN[n](g) \land \fdAnHierN[0](f)\\*
        &\qquad \limp (\fdSat^{\prime}_n(e_v^x,g) \limp \fdSat^{\prime}_0(e_v^x,f))))
        \intertext{where $\fdconjs$ is the $\HAP$-term representing the primitive recursive function}
        \conjs(\langle e \rangle) &= e\\
        \conjs(\langle h, t\rangle) &= \conjs(t) \land h\text,
      \end{align*}
      in accordance with Proposition \ref{prp:PRinHAP}.

      $\fdSat^{\prime\prime}_{n + 1}$ is $\AnHierC_{n + 1}$, whence $\fdSat^{\prime}_{n + 1}$ is the $\AnHierN_{n + 1}$-formula $\lambda_{n + 1}(\fdSat^{\prime\prime}_{n + 1})$.
    \end{itemize}

    We now define $\fdSat_n(e,F)$ as $\fdSat^{\prime}_n(e,\lambda_n \app F)$, which is compositional on $\AnHierC_n$-formulae. This, and the intermediate facts that $\fdSat^{\prime}_n$ respects conjunction, only depends on the value of $e$ for the free variables of $F$ and `commutes' with substitution, that is
    \begin{align*}
      \HAP &\prf \forall F,G,e (\fdAnHierN[n](F) \land \fdAnHierN[n](G) \limp (\fdSat^{\prime}_n(e,\fdlambda_n([F \land G])) \liff \fdSat^{\prime}_n(e,F) \land \fdSat^{\prime}_n(e,G)))\text,\\
      \HAP &\prf \forall F,e,f (\fdAnHierN[n](F) \land \forall x < F (\fdFV(F,x) \limp e \app x \= f\app x) \limp (\fdSat^{\prime}_n(e,F) \liff \fdSat^{\prime}_n(f,F)))
      \intertext{and}
      \HAP &\prf \forall F,x,s,e (\fdAnHierN[n](F) \land \fdVar(x) \land \fdTerm(s) \land (\forall v < F \neg(\fdBV(F,v) \land \fdFV(s,v))))\\
      &\qquad \limp (\fdSat^{\prime}_n(e,\fdsub \app F \app x \app s) \liff \fdSat_n^{\prime}(e_x^{\fdval \app e \app s},F))\text,
    \end{align*}
    are now shown by induction on $n$ as before, with inner inductions on the structure of $F$.
  \end{proof}
\end{Thm}

\begin{Cor}\label{cor:Sat}
  Let $\phi$ be a $\AnHierC_n$ formula with free variables among $\fv_0,\dotsc,\fv_{k - 1}$. Then $\HAP \prf \forall e ( \fdSat_n(e,\GN{\phi}) \liff \phi(e \app \GN{\fv_0}, \dotsc , e \app \GN{\fv_{k-1}} ))$.
  \begin{proof}
    Outer induction on $n$, inner induction on $\phi$.
  \end{proof}
\end{Cor}

% -------------------------------
% -------------------------------
\section{Fixpoint theories}
\label{sec:Fix}
% -------------------------------
% -------------------------------

We return now to our main item of study: the intuitionistic theory of fixpoints over Heyting arithmetic, $\IID1$. Similar to the way $\HA$ is included in $\HAP$ we want to embed $\IID1$ in a corresponding theory in the logic of partial terms, extending the notion of strictly positive fixpoints to $\HAP$. We denote this theory as $\IIDP1$. We start by formally introducing these fixpoint theories.

\subsection[$\IID1$ and $\IIDP1$]{$\bIID1$ and $\bIIDP1$}
\label{sec:IDP}

Let $\L$ be a fixed language extending $\LHA$. Let $P = (P_n)_{n \in \N}$ be a sequence of new predicate symbols such that \( P_n \) has arity \( n \).
Let $\L(P_n)$ be the language $\L$ with $P_n$ added and $\LP$ that of $\L$ with all $P_n$ added.
We will by $\POS_{P_n}(\L)$ denote the set of formulae in \( \L(P_n) \) whose free variables are \emph{exactly} $\fv_0, \dotsc, \fv_{n - 1}$  and such that $P_n$ occurs at least once in the formula and in strictly positive position only. That is, \( P_n \) may not occur in the antecedent of any implication.
Let $\POS_P(\L) = \bigcup_{n \in \N} \POS_{P_n}(\L)$. Note that this definition ensures that $\POS_P(\L)$ is a union of disjoint classes, so exactly one of the $P_n$s occur in any formula of $\POS_P(\L)$; the requirement that $P_n$ actually occur is no real restriction on the $\POS_{P_n}(\L)$. Formulae of $\POS_P(\L)$ will be used as operators with $P_n$ as parameters to generate (axioms for) strictly positive fixpoints, the \emph{strictly positive operator forms (in $\L$)} from the introduction. Now we can formally define $\IID1$ as follows: For each $\Phi \in \POS_{P_n}(\L)$ we introduce a new $n$-ary predicate $\I_{\Phi}$ and the axiom
\begin{align}\label{eq:IIDAx}
  \forall \vec{x} (\I_{\Phi}(\vec{x}) \liff \Phi(\I_{\Phi};\vec{x}))\text. \tag{$\Ax_{\Phi}$}
\end{align}
Similarly to above, we will by $\L(\I_{\Phi})$ denote the language $\L$ extended with the (specific) symbol $\I_{\Phi}$ and by $\LID$ that of $\L$ with $\I_{\Phi}$ added for every $\Phi \in \POS_P(\L)$. Then $\IID1(\Gamma)$ is $\HA(\LID)$ plus all \eqref{eq:IIDAx} for $\Phi \in \POS_P(\L) \cap \Gamma$, where $\Gamma$ is some class of $\LP$-formulae. $\ID1(\Gamma)$ is $\IID1(\Gamma) + \TND$ (or equivalently, $\PA(\LID) + \{\Ax_{\Phi} \sep \Phi \in \POS_P(\L) \cap \Gamma\}$). Often we will use this for classes $\Gamma$ whose exact definition depends on the ambient language (like $\ID1(\Pi_2)$ in the introduction), whence it is tacitly understood that the definition used for $\Gamma$ is the one appropriate for $\LP$ (i.e.\ $\Pi_2$ here means all $\Pi_2$-formulae in the language $\LP$); if there is reason to avoid confusion we will write $\Gamma(\LP)$ for this set. In case $\Gamma$ is all of $\LP$, and/or the language $\L$ is evident, we will often drop them from the notations. Thus $\IID1$ is $\IID1(\LPHA)$.

We now extend this definition to fixpoint theories over $\HAP$.

\begin{Def}\label{def:IIDP}
  The theory $\IIDP1$ is the $\LIDHAP$-theory in $\LPT$ axiomatised by
  \begin{enumerate}
  \item the axioms of $\HAP(\LIDHAP)$,\label{ax:IIDP1}
  \item the schema \eqref{eq:IIDAx} for all $\Phi \in \POS_P(\LHAP)$.\label{ax:IIDP3}
  \end{enumerate}
  The theory $\IIDP1(\Gamma)$, where $\Gamma$ is a subset of $\LPHAP$, is axiomatised by clause \ref{ax:IIDP1} above, and
  \begin{enumerate}\setcounter{enumi}2
  \item the schema \eqref{eq:IIDAx} for all $\Phi \in \POS_P(\LHAP) \cap \Gamma$.
  \end{enumerate}
\end{Def}

It follows from these definitions that the language of $\IID1$ is contained in that of $\IIDP1$; $\LIDHA \subset \LIDHAP$. Given Lemma \ref{lma:HA->HAP}, we can in fact make a stronger claim.

\begin{Lma}\label{lma:ID->IDP}
  $\IIDP1$ proves every theorem of $\IID1$.
  \begin{proof}
    By Lemma \ref{lma:HA->HAP} $\HAP(\LIDHAP) \prf \HA(\LIDHA)$. Hence by the same lemma all we need to check are the fixpoint axioms of $\IID1$. Let $\I_{\Phi}(\vec{\fv}) \liff \Phi(\I_{\Phi};\vec{\fv})$ be such an axiom. Since this is literally an axiom of $\IIDP1$ we are done.
  \end{proof}
\end{Lma}

The corresponding lemma holds for $\IID1(\Gamma)$ and $\IIDP1(\Gamma)$ if $\Gamma$ is a set of formulae (but recall that e.g.\ $\AnHierC$ and $\AnHierC_n$ below denote different sets in $\LPHA$ and $\LPHAP$).

% -------------------------------
% -------------------------------
\subsection[Partial realizability interpretation of $\IIDP1$]{Partial realizability interpretation of $\bIIDP1$}
\label{sec:rIID}
% -------------------------------
% -------------------------------

The notion of realizability we are employing (section \ref{sec:HAPr}) syntactically transforms a formula into a negative formula (Lemma \ref{lma:reaFV}). This transformation also preserves strictly positive occurrences of symbols and subformulae in many important instances. In particular, for \( \Phi \in \POS_{P_n}(\L) \), a strictly positive operator form in the \( n \)-ary predicate \( P_n \), its `realization' $x \rea \Phi$ will be in \( \POS_{\rea P_n}(\L) \), i.e.\ a strictly positive operator form in the \( (n + 1) \)-ary predicate $\rea P_n$. Since a fixpoint axiom of $\IIDP1$ will be realized if we can provide functions transforming realizers of the fixpoint predicate into realizers of the corresponding operator applied to that predicate and vice versa, if we postulate these functions to be identity we get that $\I_{\Phi}(\vec{\fv}) \liff \Phi(\I_{\Phi};\vec{\fv})$ has a realizer iff
\begin{align}\label{eq:reaHAx}
  \fv_n \rea \I_{\Phi}(\vec{\fv}) \liff \fv_n \rea \Phi(\I_{\Phi};\vec{\fv})
\end{align}
holds. By above this biconditional is itself (equivalent to) a strictly positive fixpoint axiom,
and since realizability is a transformation into negative formulae, \eqref{eq:reaHAx} is actually equivalent to an axiom of $\IIDP1(\AnHierC)$. In this section we make this realizability interpretation of $\IIDP1$ in $\IIDP1(\AnHierC)$ precise (Theorem \ref{thm:rea->Lambda}).

First, we augment the realizability transformation with a clause for the `fixpoint generators' $P$, and subsequently for all fixpoint predicates $\I_{\Phi}$, in the manner of \cite[§6]{Buchholz:1981}.

\begin{Def}\label{def:reaP}
  For the language $\LPHAP$, we stipulate $\rea P_n = P_{n + 1}$, that is
  \begin{align*}
    \tau \rea P_n(\vec{x}) \text{ is } P_{n + 1}(\vec{x},\tau)
  \end{align*}
  in the final clause of the definition of realizability.
\end{Def}

\begin{Lma}\label{lma:reaPOS}
  In $\LPHAP$ with Definition \ref{def:reaP} we have that $\fv_n \rea \Phi$ for $\Phi \in \POS_{P_n}(\LHAP)$ is in $\POS_{P_{n + 1}}(\LHAP) \cap\, \AnHierC$. More generally, let $\L$ extend $\LHAP$ with a set of relation symbols, \( R \) being among them. Suppose $\rea R$ is not the same as $\rea Q$ for any other relation $Q$.
  If $\varphi \in \L$ has only strictly positive occurrences of $R$, $\fv_n \rea \varphi$ is negative and has only strictly positive occurrences of $\rea R$, the number of which equals the number of occurrences of $R$ in $\varphi$.
  \begin{proof}
    We show by induction on complexity that for any formula $\phi$ of $\L$ where $R$ occurs only strictly positively (if at all) and any term $\tau$, $\tau \rea \phi$ has only strictly positive occurrences of $\rea R$, the number of which equals the number of occurrences of $R$ in $\phi$, and $\mathrm{FV}(\tau \rea \phi) = \mathrm{FV}(\phi) \cup \mathrm{FV}(\tau)$ if $R$ occurs in $\phi$.
    \begin{itemize}
    \item The atomic cases are trivial.
    \item The inductive steps for the Boolean connectives are straightforward, so we present only the argument for the implication.

      Thus consider the case when $\phi$ is an implication $\psi \limp \theta$, where $\psi$ does not contain $R$, $\theta$ has only strictly positive occurrences of $R$ and for all terms $\sigma_1$ and $\sigma_2$ do we have that $\sigma_1 \rea \psi$ has no occurrence of $\rea R$, the number of occurrences of $\rea R$ in $\sigma_2 \rea \theta$ is the same as the number of occurrences of $R$ in $\theta$, and if $\theta$ has at least one occurrence of $R$ then the free variables of $\sigma_2 \rea \theta$ are those of $\theta$ and $\sigma_2$.

      Let $\tau$ be a term. Then
      \begin{align*}
        \tau \rea (\psi \limp \theta) = \forall x (x \rea \psi \limp (\tau x\den \land\, \tau x \rea \theta))
      \end{align*}
      has only strictly positive occurrences of $\rea R$ by induction hypothesis, since all of these are in $\tau x \rea \theta$. Thus the number of such occurrences equals the number of such occurrences in $\tau x \rea \theta$, which by induction hypothesis is the number of occurrences of $R$ in $\theta$, i.e.\ in $\psi \limp \theta$. If $R$ occurs in $\psi \limp \theta$ then it occurs in $\theta$, whence
      \begin{align*}
        \mathrm{FV}(\tau \rea (\psi \limp \theta))
        &= \mathrm{FV}(\forall x (x \rea \psi \limp (\tau x\den \land\, \tau x \rea \theta)))\\
        &= (\mathrm{FV}(\psi) \cup ((\mathrm{FV}(\tau) \cup \mathrm{FV}(\theta)) \cup \{x\})) \setminus \{x\}\\
        % &= \mathrm{FV}(\psi) \cup \mathrm{FV}(\theta) \cup (\mathrm{FV}(\tau) \setminus \{x,y\})\\
        &= \mathrm{FV}(\psi) \cup \mathrm{FV}(\theta) \cup \mathrm{FV}(\tau)\\
        &= \mathrm{FV}(\psi \limp \theta) \cup \mathrm{FV}(\tau)
      \end{align*}
      by Lemma \ref{lma:reaFV} and induction hypothesis, since $x$ does not occur in $\psi$, $\theta$ and $\tau$.
    \item Since the existential case contains all essential ideas from the universal case but is slightly more involved, we omit the universal case.

      Consider the case when $\phi$ starts with an existential quantifier, $\exists x \psi$, where $\psi$ only has strictly positive occurrences of $R$ and for any term $\sigma$ and formula $\theta$ with only strictly positive occurrences of $R$ and the same complexity as $\psi$, $\sigma \rea \theta$ is strictly positive in $\rea R$, has as many occurrences of $\rea R$ as $\theta$ has of $R$ and satisfies $\FV(\sigma \rea \theta) = \FV(\theta) \cup \FV(\sigma)$ if $R$ occurs in $\theta$.

      Let $\tau$ be a term. By convention we can assume its variables are not bound in $\exists x \psi$. Then
      \begin{align*}
        \tau \rea (\exists x \psi) = \copr \tau \rea (\psi(x/\copl \tau))
      \end{align*}
      Since $\psi(x/\copl \tau)$ has the same complexity as $\psi$ the induction hypothesis applies, whence $\copr \tau \rea\, (\psi(x/\copl \tau))$ is strictly positive in $\rea R$ and has as many occurrences of $\rea R$ as $\psi(x/\copl \tau)$, and hence $\exists x \psi$, has of $R$. Also, if $R$ occurs in $\exists x \psi$ it occurs in $\psi(x/\copl \tau)$ whence
      \begin{align*}
        \FV(\tau \rea \exists x \psi)
        &= \FV(\copr \tau \rea (\psi(x/\copl \tau)))\\
        &= \FV(\psi(x/\copl \tau)) \cup \FV(\copr \tau)\\
        &= \FV(\psi) \setminus \{x\} \cup \FV(\tau)\\
        &= \FV(\exists x \psi) \cup \FV(\tau)
      \end{align*}
      by induction hypothesis.
    \end{itemize}
    The result now follows with Lemma \ref{lma:reaFV}.
  \end{proof}
\end{Lma}

Another consequence of the above is that $x \rea \Phi(\I_{\Phi};\vec{y})$ is strictly positive in $\rea \I_{\Phi}$, the particular choice of which we can now specify.

\begin{Def}\label{def:reaIID}
  For the language $\LIDHAP$, given $\Phi \in \POS_{P_n}(\LHAP)$ we stipulate $\rea \I_{\Phi} = \I_{\fv_n \rea \Phi}$,
  the almost negative strictly positive operator form from Lemma \ref{lma:reaPOS}, that is
  \begin{align*}
    \tau \rea \I_{\Phi}(\vec{x}) \text{ is } \I_{\fv_n \rea \Phi}(\vec{x},\tau)
  \end{align*}
  in the final clause of the definition of realizability.
\end{Def}

\begin{Thm}\label{thm:rea->Lambda}
  For the realizability interpretation given by Definitions~\ref{def:real} and \ref{def:reaIID}, $\IIDP1 \prf \varphi \Rightarrow \IIDP1(\AnHierC) \prf \exists x (x \rea \varphi)$ for $\varphi \in \LIDHAP$.
  \begin{proof}
    By Lemma \ref{lma:int->rea} and compactness it is enough to show that $\IIDP1(\AnHierC) \prf \exists x (x \rea \varphi)$ for each non-logical axiom $ \varphi $ of $\IIDP1$. Now, every axiom except the induction and fixpoint schemata is covered by \cite[VII.\ Theorem 1.5]{Beeson:1985}. As for the induction axiom $\mathrm{Ind}_{\varphi}$,
    \begin{align*}
      \HAP \vdash \cor \rea \mathrm{Ind}_{\varphi}
    \end{align*}
    For the fixpoint axioms, we observe the following.
    By Definition \ref{def:reaIID} realizability is a transformation of the language $\LIDHAP$ into the fragment defined by the almost negative strictly positive operator forms. Thus given a fixpoint axiom $\forall \vec{y} ( \I_{\Phi}(\vec{y}) \liff \Phi(\I_{\Phi};\vec{y})) $ with $\Phi \in \POS_{P_n}(\LHAP)$ let $\Psi \in \POS_{P_{n + 1}}(\LHAP)$ be the operator form from Lemma \ref{lma:reaPOS} such that $\rea \I_{\Phi} = \I_{\Psi}$.
    Then
    \begin{align*}
      \IIDP1(\AnHierC) \vdash x \rea \I_{\Phi}(\vec{y})
      &\liff \I_{\Psi}(\vec{y},x) &&\text{definition}\\
      &\liff \Psi(\I_{\Psi};\vec{y},x) &&\text{fixpoint axiom}\\
      &\liff \Psi(\vec{y},x)(\I_{\Psi}) &&\text{definition}\\
      &\liff (\fv_n \rea \Phi)(\vec{y},x)(\I_{\Psi}) &&\text{Lemma~\ref{lma:reaPOS}}\\
      &\liff (x \rea \Phi(\vec{y}))(\I_{\Psi})  &&\text{Lemma~\ref{lma:readist}}\\
      &\liff (x \rea \Phi(\vec{y}))(\rea \I_{\Phi}) &&\text{definition}
    \end{align*}
    Appealing to Lemma \ref{lma:reasubst} this equivalence, with $x$ universally quantified, is equivalent to $\cop \app \coid \app \coid \rea (\I_{\Phi}(\vec{y}) \liff \Phi(\I_{\Phi};\vec{y}))$. 
    Hence
    \begin{align*}
      \IIDP1(\AnHierC) \prf (\lambda \vec{y}. \cop \app \coid \app \coid) \rea \forall \vec{y} (\I_{\Phi}(\vec{y}) \liff \Phi(\I_{\Phi};\vec{y}))\text,
    \end{align*}
    which concludes the proof.
  \end{proof}
\end{Thm}

We close the section by proving that \( \IIDP{1}(\AnHier) \) is a conservative extension of \( \HAP \). The proof utilises the partial satisfaction predicates introduced in section~\ref{sec:Lambda}.

\begin{Thm}\label{thm:IIDP->HAP}
  $\IIDP1(\AnHierC)$ is interpretable in $\HAP$, keeping $\LHAP$ fixed.
  \begin{proof}
    We define an interpretation $\mathcal{F} \colon \LIDHAP \longrightarrow \LHAP$ which fixes $\LHAP$.
    To that extent, for $k \in \N$ let $T_k$ be the term $\coid_{\overrightarrow{\GN{\fv_{<k}}}}^{\vec{\fv}_{< k}}$ (the identity combinator modified so that $\GN{\fv_j}$ is mapped to (the value of) $\fv_j$), so that $T_k$ has exactly $\fv_0, \dotsc, \fv_{k - 1}$ free, $\HAP \prf T_k \den$ and $\HAP \prf T_k \app \GN{\fv_j} \= \fv_j$ for all $j < k$. Take an almost negative $\Phi \in \POS_{P_k}(\LHAP)$ and let $n$ be the smallest index for which $\Phi \in \AnHierC_n$, by Lemma \ref{lma:Lambdatelescope}. Set $\phi = \Phi(\lambda \vec{\fv}_{<k}.\fdSat_n(T_k,\fv_k);\vec{\fv}_{<k})$. 
    By Theorem \ref{thm:Sat} $\phi \in \LHAP$, and by Lemma \ref{lma:PosInsert} \( \phi \in \AnHierC_n \).
    Diagonalisation (Lemma \ref{lma:Diag}) yields a formula $\psi \in \AnHierC_n$ with free variables $\fv_0, \dotsc, \fv_{k - 1}$ and
    \begin{align*}
      \HAP \prf \psi(\vec{x}) \liff \phi(\vec{x},\GN{\psi})\text.
    \end{align*}
    Defining $\mathcal{F}(\I_{\Phi}) = \psi$ we thus see
    \begin{align*}
      \HAP \prf \mathcal{F}(\I_{\Phi}(\vec{x})) &\liff \psi(\vec{x})\\
      &\liff \phi(\vec{x},\GN{\psi})\\
      &\liff \Phi(\lambda \vec{\fv}_{< k}. \fdSat_n(T_k,\GN{\psi});\vec{x})\\
      &\liff \Phi(\lambda \vec{\fv}_{< k}. \psi(\vec{\fv}_{< k});\vec{x})\\
      &\liff \Phi(\psi;\vec{x})\\
      &\liff \Phi(\mathcal{F}(\I_{\Phi});\vec{x})\\
      &\liff \mathcal{F}(\Phi(\I_{\Phi};\vec{x}))
    \end{align*}
    by Corollary \ref{cor:Sat}. Setting $\mathcal{F}(\Psi) = \falsum$ for $\Psi \in \POS_P(\LHAP) \setminus \AnHierC$ (since these are inconsequential), the interpretation of the remaining axioms of $\IIDP1(\AnHierC)$ are immediate.
    Consequently $\HAP \prf \mathcal{F}[\IIDP1(\AnHierC)]$.
  \end{proof}
\end{Thm}

% -------------------------------
% -------------------------------
\section{Conclusion}
\label{sec:Conc}
% -------------------------------
% -------------------------------

Tying together the above results, we can prove the wanted conservativity result.

\begin{Thm}\label{thm:cons}
  $\IID1$ is conservative over $\HA$.
  \begin{proof}
    Let $\varphi$ be a sentence of $\LHA$. Then we have the following chain of implications:
    \begin{align*}
      \IID1 \prf \varphi 
      &\Rightarrow \IIDP1 \prf \varphi &&\text{Lemma~\ref{lma:ID->IDP}}\\
      &\Rightarrow \IIDP1(\AnHierC) \prf \exists x (x \rea \varphi) &&\text{Theorem \ref{thm:rea->Lambda}}\\
      &\Rightarrow \HAP \prf \exists x (x \rea \varphi) &&\text{Theorem \ref{thm:IIDP->HAP}}\\
      &\Rightarrow \HA \prf \varphi &&\text{Proposition~\ref{prp:HAPe->HA}}
    \end{align*}
  \end{proof}
\end{Thm}

The argument we have presented combines two interpretations: a realizability interpretation of a fixpoint theory into the subtheory of fixpoints for almost negative operator forms and a direct interpretation of this subtheory in Heyting arithmetic.
Full conservativity does not directly follow from these reductions because realizability and provability do not coincide for Heyting arithmetic: there are statements in the language of arithmetic which are (Kleene-) realizable yet not provable.
As the two notions agree on the fragment of almost negative formulae, conservativity for this class of sentences holds.
But by lifting the two interpretations to the corresponding theories formulated over Beeson's richer logic of partial terms the stronger conservation result becomes a consequence of van den Berg and van Slooten's theorem; for arithmetic formulae, provability in Heyting arithmetic coincides with realizability in Heyting arithmetic with partial terms.

A key ingredient in our argument is the existence of partial satisfaction predicates for a formula-hierarchy which exhausts the almost negative formulae.
These partial satisfaction predicates do not rely on the presence of partial terms in any essential way, and arithmetic counterparts are available which, provably in \( \HA \), are compositional truth predicates for the arithmetic formulae of the corresponding level.
A natural question is whether there exist partial satisfaction predicates for classes of formulae including the positive connectives \( \lor \) and \( \exists \), whether in the language with partial terms or pure arithmetic.
To the authors' knowledge the answer to this question is unknown, which is somewhat surprising given the importance of partial truth predicates in the proof theory of classical systems of arithmetic.

There can, however, not be a hierarchy of arithmetic formulae which (i) is closed under substitution in strictly positive positions (Lemma \ref{lma:PosInsert}), (ii) is closed under diagonalisation (Lemma \ref{lma:Diag}), (iii) exhausts the \emph{entire} language $\LHA$, and (iv) has partial satisfaction predicates for every level.
Such a hierarchy with corresponding satisfaction predicates would give rise to a direct interpretation of $\IID1$ in $\HA$ by generalising Theorem \ref{thm:IIDP->HAP}, but it would also allow the classical theory $\ID1$ to be interpreted in $\PA$ by the same argument.
This cannot be the case because $\ID1$ is strictly stronger than $\PA$. 
On the other hand, this line of reasoning shows that $\ID1(\AnHierC)$ \emph{is} conservative over $\PA$, again by generalising Theorem \ref{thm:IIDP->HAP}. Thus, the \emph{strictly} positive almost negative operator forms gives a weaker classical fixpoint theory than the merely positive almost negative ones, since the latter is equivalent to $\ID1$. We see that trying to carry out our argument for this class of fixpoints fails on Lemma \ref{lma:PosInsert}, since $\AnHierC_n$ is not closed under substitutions into merely positive positions.
A final note is that, since $\Pi_1$ has a $\Pi_1$ satisfaction predicate and is closed under substitutions into strictly positive positions, $\ID1(\Pi_1)$ is conservative over $\PA$ essentially by replacing $\AnHierC_n$ by $\Pi_1$ in the proof of Theorem \ref{thm:IIDP->HAP}. On the other hand $\ID1(\Pi_2)$ is not conservative over $\PA$, since it interprets the Kripke-Feferman theory of truth $\KF$ (Feferman, \cite{feferman:1991}) which proves the consistency of $\PA$.
It would therefore be interesting to identify the subclasses $\Gamma$ of $\Pi_2$ forms for which $\ID1(\Gamma)$ is still conservative.

We do believe this argument generalises to finite iterations $\IID{n}$ of fixpoint theories, however. The main idea would then be to repeat the application of Theorem \ref{thm:IIDP->HAP} inductively to reduce $\IIDP{n + 1}(\AnHierC)$ to $\IIDP{n}(\AnHierC)$, viewing $\HAP$ as $\IIDP0(\AnHierC)$ ($\IIDP{n}(\AnHierC)$ should mean \emph{all} operator forms are almost negative, not just the $n$th `top layer'). The possible issue we see in section \ref{sec:SatHAP} is that Theorem \ref{thm:Sat} requires (primitive) recursive transformations from $\AnHierC_n$ to $\AnHierN_n$ as in Lemma \ref{lma:Lambdatransform}, which is not as straightforward for countable vocabularies. This should be manageable by more careful attention to coding, or finally by a compactness argument in \ref{thm:IIDP->HAP}. The proof of Theorem \ref{thm:IIDP->HAP} should otherwise generalise, essentially by placing an `$\mathcal{F}$' defined by induction hypothesis in front of every appearance of $\Phi$ except the first.
In the rest of the argument of Theorem \ref{thm:cons}, generalising in this direction would add a countable number of relation symbols and axioms governing these, and Lemma \ref{lma:ID->IDP} clearly generalises to this case. For the results leading up to Theorem \ref{thm:rea->Lambda}, prominently Lemma \ref{lma:reaPOS}, we have taken some care not to specify the relation symbols in the base language, so reformulating Theorem \ref{thm:rea->Lambda} slightly this step should also generalise by induction. Finally, the final step would remain unchanged.

A natural question would be whether a generalisation of this argument works also for some transfinite iterations $\IID{\alpha}$. While we have no clear ideas in this direction, it's a reasonable question for future work.

\bibliographystyle{plain}
\bibliography{BibliographyConsIID}
\end{document}